\documentclass[12pt]{article}
\usepackage{amsthm,amsfonts,amssymb,amsmath}
\usepackage{stmaryrd}

\usepackage{cite,hyperref}
\usepackage{epsfig}
\usepackage{url}
\usepackage{xcolor,tikz}
\usetikzlibrary{matrix}
\usepackage{multicol}
\usepackage{fullpage}
\usepackage{bbm}

\usepackage[shortlabels]{enumitem}

\numberwithin{equation}{section}

\setlist{leftmargin=3\parindent,labelindent=3\parindent}
\setlist[enumerate]{%
  leftmargin=3\parindent,%
  align=left,%
  labelwidth=3\parindent,%
  labelsep=0pt%
}
\setlist[enumerate,1]{% 
  label={\normalfont (\thesection.\arabic{equation})}, ref={\normalfont \thesection.\arabic{equation}},
  resume%
}

\newtheorem{thm}[equation]{Theorem}
\newtheorem{cor}[equation]{Corollary}
\newtheorem{lem}[equation]{Lemma}
\newtheorem{prop}[equation]{Proposition}
\newtheorem{claim}[equation]{Claim}

\newtheorem{conj}[equation]{Conjecture}
\newtheorem{ques}[equation]{Question}

\theoremstyle{definition}
\newtheorem{defn}[equation]{Definition}
\newtheorem{ex}[equation]{Example}

\newtheorem{rem}[equation]{Remark}
\newtheorem{obs}[equation]{Observation}

\theoremstyle{remark}

\title{Disconnected Common Graphs via Supersaturation}
\author{Jae-baek Lee\thanks{Department of Mathematics, Yonsei University, Seoul, Republic of Korea. E-mail: \texttt{dlwoqor0923@gmail.com}. This work was completed when the first author was affiliated with the University of Victoria.}
\and 
Jonathan A. Noel\thanks{Department of Mathematics and Statistics, University of Victoria, Victoria, B.C., Canada. E-mail \texttt{noelj@uvic.ca}. Research supported by NSERC Discovery Grant RGPIN-2021-02460, NSERC Early Career Supplement DGECR-2021-00024 and a Start-Up Grant from the University of Victoria.}}

\DeclareTextCompositeCommand{\v}{OT1}{l}{l\nobreak\hspace{-.1em}'}
\DeclareTextCompositeCommand{\v}{OT1}{t}{t\nobreak\hspace{-.1em}'\nobreak\hspace{-.15em}}

\DeclareMathOperator{\UC}{UC}

\begin{document}

\maketitle

\begin{abstract}
A graph $H$ is said to be \emph{common} if the number of monochromatic labelled copies of $H$ in a $2$-colouring of the edges of a large complete graph is asymptotically minimized by a random colouring. It is well known that the disjoint union of two common graphs may be uncommon; e.g., $K_2$ and $K_3$ are common, but their disjoint union is not. We investigate the commonality of disjoint unions of multiple copies of $K_3$ and $K_2$. As a consequence of our results, we obtain an example of a pair of uncommon graphs whose disjoint union is common. Our approach is to reduce the problem of showing that certain disconnected graphs are common to a constrained optimization problem in which the constraints are derived from supersaturation bounds related to Razborov's Triangle Density Theorem. We also improve bounds on the Ramsey multiplicity constant
of a triangle with a pendant edge and the disjoint union of $K_3$ and $K_2$. 
\end{abstract}

\section{Introduction}

In one of the first applications of the probabilistic method, Erd\H{o}s~\cite{Erdos47} showed that a random colouring of the edges of a clique on $(1-o(1))2^{-1/2}e^{-1}k2^{k/2}$ vertices with red and blue contains no monochromatic complete graph on $k$ vertices with positive probability; this implies a lower bound on the \emph{Ramsey number} of the complete graph $K_k$, i.e. the smallest $N$ for which every $2$-colouring of the edges of $K_N$ contains a monochromatic $K_k$. To this day, Erd\H{o}s' bound has been improved only slightly by Spencer~\cite{Spencer75}. One of the core themes in Ramsey theory is that random colourings tend to perform well in avoiding certain monochromatic substructures. 

This intuition extends to the closely related area of ``Ramsey multiplicity'' in which the goal is to minimize the number of monochromatic labelled copies of a given graph $H$ in a red/blue colouring of the edges of $K_N$ asymptotically as $N$ tends to infinity. A graph $H$ is said to be \emph{common} if this asymptotic minimum is achieved by a sequence of random colourings. A famous result of Goodman~\cite{Goodman59} implies that $K_3$ is common (see Theorem~\ref{th:GoodmanRamsey}). Inspired by this, Erd\H{o}s~\cite{Erdos62} conjectured that $K_k$ is common for all $k$ and, nearly two decades later, Burr and Rosta~\cite{BurrRosta80} conjectured that every graph $H$ is common. Sidorenko~\cite{Sidorenko89} observed that the \emph{paw graph} consisting of a triangle with a pendant edge is uncommon. Around the same time, Thomason~\cite{Thomason89} showed that $K_k$ is uncommon for all $k\geq4$; thus, the aforementioned conjectures are both false. Later, Jagger, S\v{t}ov\'i\v{c}ek and Thomason~\cite{JaggerStovicekThomason96} proved that every graph $H$ containing $K_4$ is uncommon. In particular, almost every graph is uncommon. In recent years, there has been a steady flow of results proving that the members of certain families of graphs are common or uncommon~\cite{Kral+22,KralVolecWei22+,GrzesikLeeLidickyVolec22,Behague+First,Behague+Second,CsokaHubaiLovasz23,KoLee23,Hatami+12,HancockKralKrncVolec23,CummingsYoung11,Raghuvanshi16,KimLee24}. In spite of this, the task of classifying common graphs seems hopelessly difficult. 

The main goal of this paper is to provide an approach for bounding the number of monochromatic copies of certain disconnected graphs in a colouring of $K_N$ and to use it to obtain several new families of common graphs. Given graphs $H_1$ and $H_2$, let $H_1\sqcup H_2$ denote their disjoint union; also, for a graph $F$ and $\ell\geq1$, let $\ell\cdot F$ be the disjoint union of $\ell$ copies of $F$. The argument of Sidorenko~\cite{Sidorenko89} that the paw graph is uncommon also shows that $K_3\sqcup K_2$ is uncommon (with the same proof). Most of our results focus on the commonality of unions of several copies of $K_3$ and $K_2$. Our first result is as follows. 

\begin{thm}
\label{th:DE(K3cupK3)}
For $0\leq \ell\leq2$, the graph $(2\cdot K_3)\sqcup (\ell\cdot K_2)$ is common. 
\end{thm}

We also show that this is best possible in the sense that $(2\cdot K_3)\sqcup (3\cdot K_2)$ is uncommon; see Proposition~\ref{prop:DE(K_3cupK_3)upper}. Since $K_3$ and $K_2$ are both common, Sidorenko's result~\cite{Sidorenko89} that $K_3\sqcup K_2$ is uncommon tells us that the disjoint union of two common graphs can be uncommon. Using Theorem~\ref{th:DE(K3cupK3)}, we find that the opposite phenomenon is also possible; the disjoint union of two uncommon graphs can be common. In fact, the disjoint union of two copies of a single uncommon graph can be common.

\begin{cor}
\label{cor:uncommonUnion}
There exists an uncommon graph $H$ such that $H\sqcup H$ is common. 
\end{cor}

\begin{proof}
Consider $H=K_3\sqcup K_2$. The fact that $H$ is uncommon was shown by Sidorenko~\cite{Sidorenko89}, and the fact that $H\sqcup H$ is common follows from Theorem~\ref{th:DE(K3cupK3)} with $\ell=2$. 
\end{proof}

We remark that our results also allow us to obtain new examples of graphs $H_1$ and $H_2$ such that $H_1$ is common, $H_2$ is uncommon and $H_1\sqcup H_2$ is common. However, the existence of such a pair of graphs was already known; see~\cite[Subsection~1.1]{KralVolecWei22+}. We also prove a general result on disjoint unions of triangles and edges, provided that the number of triangles is at least three.

\begin{thm}
\label{th:5/3}
For $k\geq3$ and $0\leq \ell\leq 5k/3$, the graph $(k\cdot K_3)\sqcup (\ell\cdot K_2)$ is common.
\end{thm}

In fact, this is derived from a more general result (Corollary~\ref{cor:5/3TT}) which applies to the disjoint union of a Sidorenko graph (see \eqref{eq:Sid} below for a definition) and a graph built up from gluing together triangles in a tree-like fashion. While we do not believe that the bound on $\ell$ in Theorem~\ref{th:5/3} is tight for all $k\geq3$, the next theorem shows that it is tight for $k=3$ and within a factor $1.1799 + o(1)$ of being tight for general $k$.

\begin{thm}
\label{th:uncommon}
For $k\geq 1$ and $\ell=\lceil 1.9665k\rceil$, the graph $(k\cdot K_3)\sqcup (\ell\cdot K_2)$ is uncommon.
\end{thm}

Several of the results in this paper are best understood in the context of graph limits. A \emph{kernel} is a bounded measurable function $U:[0,1]^2\to\mathbb{R}$ such that $U(x,y)=U(y,x)$ for all $x,y\in [0,1]$. A \emph{graphon} is a kernel such that $0\leq W(x,y)\leq 1$ for all $x,y\in [0,1]$. The set of all graphons is denoted $\mathcal{W}_0$. Given a graph $G$, let $v(G):=|V(G)|$ and $e(G):=|E(G)|$. A graph $G$ is said to be \emph{empty} if $e(G)=0$. Each graph $G$ can be associated to a graphon $W_G$ by dividing $[0,1]$ into $v(G)$ intervals $I_1,\dots,I_{v(G)}$ of equal measure corresponding to the vertices of $G$ and setting $W_G$ equal to $1$ on $I_i\times I_j$ if the $i$th and $j$th vertices are adjacent and $0$ otherwise. The \emph{homomorphism density} of a graph $H$ in a kernel $U$ is defined by 
\[t(H,U):=\int_{[0,1]^{V(H)}}\prod_{uv\in E(H)}W(x_u,x_v)dx_{V(H)}\]
where $x_{V(H)}=(x_v: v\in V(H))$. We refer the reader to~\cite{Lovasz12} for more background on graph limits. The \emph{Ramsey multiplicity constant} of a graph $H$ is defined to be
\[c(H):=\inf_{W\in\mathcal{W}_0} (t(H,W)+t(H,1-W)).\]
In this language, a graph $H$ is \emph{common} if and only if
\begin{equation}\label{eq:common}c(H)=2(1/2)^{e(H)}.\end{equation}
As stated above, $K_3\sqcup K_2$ and the paw graph $P$ are uncommon. We obtain, to our knowledge, the tightest known upper bounds on the Ramsey multiplicity constants of these two graphs; for the former graph, we also obtain a reasonably tight lower bound which is proven without the assistance of the flag algebra method; see Remark~\ref{rem:K3K2flag}. 

\begin{thm}
\label{th:oneK3}
$0.121423< c(K_3\sqcup K_2)< 0.121450$.
\end{thm}

\begin{thm}
\label{th:paw}
The paw graph $P$ satisfies $c(P)< 0.121415$.
\end{thm}

Note that, for every graph $H$ such that $c(H)$ is currently known, either $H$ is common or $c(H)$ is achieved by a ``Tur\'an graphon'' $W_{K_k}$ for some $k\geq3$~\cite{FoxWigderson23}. To our knowledge, Theorem~\ref{th:oneK3} is the closest that any result has come to determining $c(H)$ for a graph $H$ which does not fit into either of these two categories. As discussed later in Remark~\ref{rem:K3K2flag}, the lower bound in Theorem~\ref{th:oneK3} can be improved by at least $0.022\%$ using the flag algebra method; however, such a proof would most likely be verifiable only with heavy computer assistance, and is thus unlikely to provide much in terms of valuable insights.

Several of the known results on common graphs actually establish stronger inequalities than \eqref{eq:common}. Following~\cite{Behague+Second}, a non-empty graph $H$ is said to be \emph{strongly common} if
\begin{equation}\label{eq:stronglycommon}t(H,W)+t(H,1-W)\geq t(K_2,W)^{e(H)}+t(K_2,1-W)^{e(H)}\end{equation}
for every graphon $W$. A simple application of Jensen's Inequality tells us that every strongly common graph is common. A classical example of a strongly common graph is $K_3$; see Theorem~\ref{th:GoodmanRamsey}. A non-empty graph $H$ is said to be \emph{Sidorenko} if 
\begin{equation}\label{eq:Sid}t(H,W)\geq t(K_2,W)^{e(H)}\end{equation}
for every graphon $W$. Clearly, every Sidorenko graph is strongly common which, in turn, implies that every such graph is common. By taking $W=W_{K_2}$, one can see that every Sidorenko graph must be bipartite. Sidorenko's Conjecture~\cite{Sidorenko93} famously states that every bipartite graph is Sidorenko. For recent progress on Sidorenko's Conjecture, see~\cite{ConlonFoxSudakov10,ConlonLee17,ConlonLee21,ConlonKimLeeLee18,KimLeeLee16,Szegedy15,Hatami10}. Currently, every bipartite graph $H$ which is known to be common is also known to be Sidorenko. Also, the only known examples of strongly common graphs which are not Sidorenko are the odd cycles~\cite{Behague+Second,Goodman59,KimLee24}.

Our strategy for obtaining new examples of common graphs relies on strong correlation inequalities, such as \eqref{eq:stronglycommon} and \eqref{eq:Sid}. Given this, it is natural to wonder whether all common graphs are strongly common; this question was raised in~\cite{Behague+Second}. As it turns out, this is far from true. For example, $K_3\sqcup K_3$ is common but not strongly common, and there are many other examples as well. 

\begin{thm}
\label{th:commonNotStrongly}
There exists a common graph $H$ which is not strongly common. 
\end{thm}

In Section~\ref{sec:reduction}, we prove a general lower bound on $c(H)$ for graphs $H$ satisfying certain correlation inequalities involving homomorphism densities. This lower bound is written in terms of the minimum of a certain constrained optimization problem over two variables $x$ and $y$. In Section~\ref{sec:quick}, we illustrate this result by presenting a proof of the lower bound in Theorem~\ref{th:oneK3}. Our goal in Section~\ref{sec:solveOpt} is to do a bit of calculus to remove the dependence on $y$ in the optimization problem. In Section~\ref{sec:K3}, we showcase some applications of this approach involving graphs built up from triangles and edges via certain gluing operations. In particular, we prove Theorems~\ref{th:DE(K3cupK3)} and~\ref{th:5/3}. In Section~\ref{sec:uncommon}, we turn our attention to providing constructions of non-random colourings to obtain upper bounds on the Ramsey multiplicity of uncommon graphs; in particular, we prove Theorems~\ref{th:uncommon} and~\ref{th:commonNotStrongly} and the upper bounds in Theorems~\ref{th:oneK3} and~\ref{th:paw}. The last section consists of open problems and further discussion. In the appendix, we derive some simpler sufficient conditions for bounding the minimum of the optimization problem which are applied in Section~\ref{sec:K3}.

\section{A Bound on the Ramsey Multiplicity Constant}
\label{sec:reduction}

We start with the following standard observation which is very useful when dealing with homomorphism densities of disconnected graphs. 

\begin{obs}
\label{obs:union}
For any graphs $F$ and $H$ and graphon $W$, 
\[t(F\sqcup H,W)=t(F,W)t(H,W).\]
\end{obs}

The following definition describes a type of correlation inequality that appears frequently in the study of extremal problems on homomorphism densities; see, e.g.,~\cite{Lee21,GrzesikLeeLidickyVolec22,Behague+Second,BlekhermanRaymond22}.

\begin{defn}
Let $H$ and $J$ be non-empty graphs and let $k$ and $\ell$ be real numbers such that $k\neq 0$. We say that $H$ is \emph{$(J,k,\ell)$-correlated} if
\begin{enumerate}
\stepcounter{equation}
    \item\label{eq:eCondition} $e(H)=ke(J)+\ell$ and
\stepcounter{equation}
    \item\label{eq:correlation} $t(H,W)\geq t(J,W)^kt(K_2,W)^{\ell}$ for every graphon $W$.\footnote{In the case that $\ell=0$ and $t(K_2,W)=0$, we regard $0^0$ as being equal to $1$.}
\end{enumerate}
\end{defn}

Our approach to obtaining new examples of common graphs relies on certain ``supersaturation bounds'' from extremal graph theory. That is, we require a lower bound on $t(J,W)$ for a graph $J$ in terms of $t(K_2,W)$.

\begin{defn}
Given a non-empty graph $J$, let $\rho_J:[0,2]\to[0,\infty)$ be the function defined by
\[\rho_J(z):=2^{e(J)}\inf\{t(J,W): W\in\mathcal{W}_0\text{ and }t(K_2,W)=z/2\}\]
for all $0\leq z\leq 2$.
\end{defn}

We pause for a few basic observations. 

\begin{obs}
\label{obs:rhoBounds}
We have $\rho_J(1+x)\leq \left(1+x\right)^{e(J)}$ for any non-empty graph $J$ and $-1\leq x\leq 1$.
\end{obs}

\begin{proof}
Consider the constant graphon $W$ which is equal to $(1+x)/2$ everywhere. Then $\rho(1+x)\leq 2^{e(J)}t(J,W)=\left(1+x\right)^{e(J)}$. 
\end{proof}

\begin{obs}
$\rho_J(1+x)=\left(1+x\right)^{e(J)}$ for all $-1\leq x\leq 1$ if and only if $J$ is Sidorenko.
\end{obs}

\begin{proof}
If $J$ is Sidorenko and $x\in [-1,1]$, then $t(J,W)\geq \left(\frac{1+x}{2}\right)^{e(J)}$ for every graphon with $t(K_2,W)=(1+x)/2$. Combining this with the upper bound in Observation~\ref{obs:rhoBounds} yields $\rho_J(1+x)=\left(1+x\right)^{e(J)}$. If $J$ is not Sidorenko, then there exists a graphon $W$ such that $t(J,W)<t(K_2,W)^{e(J)}$. Thus, for $x=2t(K_2,W)-1$, we have $\rho_J(1+x)<\left(1+x\right)^{e(J)}$. 
\end{proof}

All of the new examples of common graphs in this paper will be built up from known examples of strongly common graphs. However, it seems to us that the same approach may also work with graphs that are not strongly common, but satisfy weaker inequalities. This motivates the next definition. 

\begin{defn}
Given a function $g:[0,2]\to [0,\infty)$, we say that a graph $H$ is \emph{$g$-bounded} if, for $-1\leq x\leq 1$, it holds that
\[2^{e(H)}\left(t(H,W)+t(H,1-W)\right)\geq g(1+x)+g(1-x)\]
for every graphon $W$ such that $t(K_2,W)=(1+x)/2$. 
\end{defn}

Let us tie this in with the notion of strongly common graphs defined earlier. 

\begin{obs}
\label{obs:z^m}
Let $H$ be a non-empty graph and let $g:[0,2]\to [0,\infty)$ be defined by $g(z)=z^{e(H)}$. Then $H$ is strongly common if and only if it is $g$-bounded.
\end{obs}

\begin{proof}
If $H$ is strongly common, then, for any $0\leq x\leq 1$ and graphon $W$ with $t(K_2,W)=(1+x)/2$, we have
\[2^{e(H)}\left(t(H,W)+t(H,1-W)\right)\geq 2^{e(H)}\left(t(K_2,W)^{e(H)}+t(K_2,1-W)^{e(H)}\right)\]
\[=2^{e(H)}\left(\left(\frac{1+x}{2}\right)^{e(H)}+\left(\frac{1-x}{2}\right)^{e(H)}\right) = g(1+x)+g(1-x)\]
and so $H$ is $g$-bounded. 

On the other hand, suppose that $H$ is $g$-bounded. Let $W$ be a graphon and define $x=2t(K_2,W)-1$. We get
\[t(H,W)+t(H,1-W)\geq 2^{-e(H)}\left(g(1+x)+g(1-x)\right) = t(K_2,W)^{e(H)}+t(K_2,1-W)^{e(H)}\]
and so $H$ is strongly common.
\end{proof}

Given a $g$-bounded graph $J$ and $\rho\leq \rho_J$, Lemma~\ref{lem:reduction} below implies a bound on the Ramsey multiplicity constant $c(H)$ in terms of $g,k,\ell$ and $\rho$ for any graph $H$ which is $(J,k,\ell)$-correlated. Before stating the lemma, we require one more definition.

\begin{rem}
In the following definition, the expressions $(1+x)^\ell$ and $(1-x)^\ell$ in the case $\ell=0$ are treated as being equal to $1$ for all $-1\leq x\leq 1$. Expressions which may evaluate to $0^0$ in a similar manner are treated in the same way throughout the paper. 
\end{rem}

\begin{defn}
Given a function $g:[0,2]\to[0,\infty)$ and real numbers $k$ and $\ell$ such that $k\neq 0$, define $f^g_{k,\ell}:[-1,1]\times \mathbb{R}\to\mathbb{R}$ by
\begin{equation}
\label{eq:fDef}
f^g_{k,\ell}(x,y):=(1+x)^{\ell}\left(g(1+x)-y\right)^k+(1-x)^{\ell}\left(g(1-x)+y\right)^k.
\end{equation}
\end{defn}

\begin{lem}
\label{lem:reduction}
Let $g,\rho:[0,2]\to[0,\infty)$ such that $\rho\leq g$, let $k$ and $\ell$ be real numbers such that $k\neq 0$ and let $J$ be a non-empty $g$-bounded graph such that $\rho_J\geq \rho$. Then, for any graphon $W$, $2^{ke(J)+\ell}\left(t(K_2,W)^\ell t(J,W)^k+t(K_2,1-W)^\ell t(J,1-W)^k\right)$ is at least the minimum of $f^g_{k,\ell}(x,y)$ over all $x$ and $y$ such that $-1\leq x\leq 1$ and $0\leq y\leq g(1+x)-\rho(1+x)$. 
\end{lem}

\begin{proof}
Define $m:=e(J)$. Let $W$ be a graphon and let $x:=2t(K_2,W)-1$. By the symmetry between $W$ and $1-W$, we may assume, without loss of generality, that
\begin{equation}\label{eq:wlog}g(1+x)-2^mt(J,W)\geq g(1-x)-2^mt(J,1-W).\end{equation}
Now, define
\[y:=\max\{0,g(1+x)-2^{m}t(J,W)\}.\]
Then, by definition of $y$, we have $y \geq g(1+x)- 2^m t(J,W)$. Equivalently,
\begin{equation}\label{eq:-y1}2^mt(J,W)\geq g(1+x)- y.\end{equation}
We claim that $y\leq g(1+x)$. Indeed, if $y=0$, then this is true because the co-domain of $g$ is $[0,\infty)$ and, if $y=g(1+x)-2^mt(J,W)$, then it is true because $t(J,W)\geq0$. So, combining this with \eqref{eq:-y1}, we get 
\begin{equation}\label{eq:-y}2^mt(J,W)\geq g(1+x)- y\geq0,\end{equation}
Since $J$ is $g$-bounded and $m=e(J)$, it holds that
\[2^m\left(t(J,W)+t(J,1-W)\right)\geq g(1+x)+g(1-x)\]
which is equivalent to
\[2^mt(J,1-W)-g(1-x) \geq g(1+x)-2^mt(J,W).\]
If $y=g(1+x)-2^mt(J,W)$, then the above inequality translates to $2^mt(J,1-W)\geq g(1-x)+y$. On the other hand, if $y=0$, then, by definition of $y$, we must have $g(1+x)-2^mt(J,W)\leq 0$. By \eqref{eq:wlog}, this implies that $g(1-x)-2^mt(J,1-W)\leq 0$. So, in the case that $y=0$, we again get $2^mt(J,1-W)\geq g(1-x)+y$. Thus, regardless of the value of $y$, we have
\begin{equation}\label{eq:+y}2^mt(J,1-W)\geq g(1-x)+ y\geq0.\end{equation}
Now, since $x=2t(K_2,W)-1$ and $t(K_2,W)+t(K_2,1-W)=1$,
\[ t(K_2,W)^{\ell}t(J,W)^k + t(K_2,1-W)^{\ell}t(J,1-W)^k=\left(\frac{1+x}{2}\right)^{\ell}t(J,W)^k + \left(\frac{1-x}{2}\right)^{\ell}t(J,1-W)^k.\]
By \eqref{eq:-y} and \eqref{eq:+y}, the right side of the above equality is at least
\[2^{-km-\ell}(1+x)^{\ell}\left(g(1+x)-y\right)^k + 2^{-km-\ell}(1-x)^{\ell}\left(g(1-x)+y\right)^k = 2^{-ke(J)-\ell}f^g_{k,\ell}(x,y).\]

To complete the proof, we need to show that $x$ and $y$ satisfy the constraints $-1\leq x\leq 1$ and $0\leq y\leq g(1+x)-\rho(1+x)$. The constraint $-1\leq x\leq 1$ follows simply from the fact that $x=2t(K_2,W)-1$ and $0\leq t(K_2,W)\leq 1$. Also, $y\geq0$ simply by definition. If $y=0$, then the constraint $y\leq g(1+x)-\rho(1+x)$ holds by the hypothesis $\rho\leq g$ of the lemma. On the other hand, if $y=g(1+x) - 2^mt(J,W)$, then, since $x=2t(K_2,W)-1$, we have
\[y=g(1+x)- 2^mt(J,W)\leq g(1+x) - 2^m\inf\{t(J,W): t(K_2,W)=(1+x)/2\}\]
\[=g(1+x) - \rho_J(1+x)\leq g(1+x)-\rho(1+x).\]
The result follows.
\end{proof}

\begin{rem}
For most graphs $J$ satisfying the hypotheses of Lemma~\ref{lem:reduction}, it seems unlikely that $t(J,W)$ would be anywhere near $2^{-e(J)}\rho(2t(K_2,W))$ for the graphon $W$ which minimizes $t(K_2,W)^\ell t(J,W)^k+t(K_2,1-W)^\ell t(H,1-W)^k$. Thus, one may be able to tighten the upper bound constraint on $y$ in certain cases (or even in general). However, while it is perhaps not the most powerful constraint possible, the bound $y\leq g(1+x)-\rho(1+x)$ is sufficient to obtain nearly sharp bounds on the Ramsey multiplicity constant of certain graphs, as we shall demonstrate in the next section.
\end{rem}

\section{A Quick Application}
\label{sec:quick}

As a simple application of Lemma~\ref{lem:reduction}, we prove the lower bound in Theorem~\ref{th:oneK3}, which we restate here for convenience. 

\begin{prop}
\label{prop:oneK3Lower}
$c(K_3\sqcup K_2)> 0.121423$.
\end{prop}

Before proving this proposition, we require a few preparations. First, we need that $K_3$ is strongly common, which essentially follows from an old result of Goodman; see, e.g.,~\cite[Theorem~5.3]{Behague+Second} for a modern proof.

\begin{thm}[Goodman's Theorem~\cite{Goodman59}]
\label{th:GoodmanRamsey}
$K_3$ is strongly common. 
\end{thm}

We also require bounds on $\rho_{K_3}(1+x)$. The precise value of this function for all $-1\leq x\leq 1$ can be derived from the landmark ``Triangle Density Theorem'' of Razborov~\cite{Razborov08}. For the purposes of proving Proposition~\ref{prop:oneK3Lower}, it will be enough to have a tight bound on $\rho_{K_3}(1+x)$ when $x$ is close to zero and a more crude bound for larger $x$. For the former, we use the following theorem which was first announced by Fisher~\cite{Fisher89}; as mentioned in~\cite{Razborov08}, the proof contained a hole that can be patched using a later result of~\cite{GoldwurmSantini00}. A new proof was found by Razborov~\cite{Razborov07} prior to proving the general Triangle Density Theorem in~\cite{Razborov08}.

\begin{thm}[Fisher~\cite{Fisher89} and Goldwurm and Santini~\cite{GoldwurmSantini00}; see also Razborov~\cite{Razborov07}]
\label{th:Fisher}
Every graphon $W$ with $t(K_2,W)\leq 2/3$ satisfies
\[t(K_3,W)\geq \frac{1}{9} \left(-2 \left(2 + \sqrt{4 - 6 t(K_2,W)}\right) + 3t(K_2,W) \left(3 + \sqrt{4 - 6 t(K_2,W)}\right)\right)\]
\end{thm}

For larger edge densities, we resort to the following convenient linear bound proved by Bollob\'as~\cite{Bollobas76} (see also~\cite[Chapter~VI.1]{Bollobas78}).

\begin{thm}[Bollob\'as~\cite{Bollobas76}]
\label{th:Bollobas}
Every graphon $W$ satisfies 
\[t(K_3,W)\geq \frac{4}{3}t(K_2,W)-\frac{2}{3}.\]
\end{thm}

We combine Theorems~\ref{th:Fisher} and~\ref{th:Bollobas} to get the following general lower bound on $\rho_{K_3}(1+x)$. 

\begin{cor}
\label{cor:Fisher}
We have
\[\rho_{K_3}(1+x)\geq \begin{cases}
0 & \text{for }-1\leq x< 0,\\
\frac{4}{9} \left(1 - \sqrt{1 - 3 x} + 3x \left(3 + \sqrt{1 - 3 x}\right)\right) & \text{for }0\leq x\leq 1/3\\ 
16x/3 & \text{for }1/3< x\leq 1.\end{cases}\].
\end{cor}

\begin{proof}
The fact that $\rho_{K_3}(1+x)\geq0$ for all $x$ is simply by definition. The bound $\rho_{K_3}(1+x)\geq 16x/3$ for all $-1\leq x\leq 1$ comes from plugging $x=2t(K_2,W)-1$ into Theorem~\ref{th:Bollobas}. Finally, for $0\leq x\leq 1/3$, plugging $x=2t(K_2,W)-1$ into Theorem~\ref{th:Fisher} yields
\[\rho_{K_3}(1+x)\geq \frac{4}{9} \left(1 - \sqrt{1 - 3 x} + 3x \left(3 + \sqrt{1 - 3 x}\right)\right).\]
The result follows. 
\end{proof}

We now present the proof of Proposition~\ref{prop:oneK3Lower}.

\begin{proof}[Proof of Proposition~\ref{prop:oneK3Lower}]
By Observation~\ref{obs:union}, we have 
\[t(K_3\sqcup K_2,W) = t(K_3,W)t(K_2,W)\]
for every graphon $W$. By Theorem~\ref{th:GoodmanRamsey} and Observation~\ref{obs:z^m}, $K_3$ is $z^3$-bounded. So, by Lemma~\ref{lem:reduction}, $c(K_3\sqcup K_2)$ is at least the minimum of 
\[2^{-4}(1+x)((1+x)^3-y) + 2^{-4}(1-x)((1-x)^3+y)\]
over all $x$ and $y$ such that $-1\leq x\leq 1$ and $0\leq y\leq (1+x)^3-\rho_{K_3}(1+x)$. Expanding the above expression yields
\begin{equation}\label{eq:oneK3}2^{-4}(1+x)^4 + 2^{-4}(1-x)^4 -2^{-3}xy.\end{equation}
It suffices to bound the expression in \eqref{eq:oneK3} below by 0.121423 for all $x$ and $y$ satisfying the constraints described above. First, if $x\leq 0$, then, since $y\geq 0$, the expression in \eqref{eq:oneK3} is at least 2. So, we assume that $x>0$. In this case, the expression in \eqref{eq:oneK3} is minimized when $y$ is as large as possible. 

If $1/3<x\leq 1$, then Corollary~\ref{cor:Fisher} tells us that the expression in \eqref{eq:oneK3} is at least
\[2^{-4}(1+x)^4 + 2^{-4}(1-x)^4 -2^{-3}x((1+x)^3-16x/3).\]
The derivative of this expression with respect to $x$ is $-\frac{1}{24}(27 x^2 - 50 x + 3)$, which is positive for all $x$ such that $1/3 \leq x\leq 1$; thus, when $1/3<x\leq 1$ and $0\leq y\leq (1+x)^3-\rho_{K_3}(1+x)$, the expression in \eqref{eq:oneK3} is bounded below by 
\[2^{-4}(1+(1/3))^4 + 2^{-4}(1-(1/3))^4 -2^{-3}(1/3)((1+(1/3))^3-16/9) = 5/27>0.121423.\]

Finally, we assume that $0\leq x\leq 1/3$. In this case, by  Corollary~\ref{cor:Fisher}, the expression in \eqref{eq:oneK3} is at least
\[2^{-4}(1+x)^4 + 2^{-4}(1-x)^4 -2^{-3}x((1+x)^3-\frac{4}{9} \left(1 - \sqrt{1 - 3 x} + 3x \left(3 + \sqrt{1 - 3 x}\right)\right)).\]
If we let $z=\sqrt{1 - 3 x}$, then the above expression can be rewritten as
\begin{equation}\label{eq:oneK3two}\frac{1}{216}(3z^6 + 4z^5 + 12z^4 -4z^3-28z^2+40).\end{equation}
Our final aim is to minimize this function over all $0\leq z\leq 1$. The derivative is
\[\frac{1}{108} z (9z^4 + 10z^3+24z^2-6z-28).\]
So, in particular, $z=0$ is a critical point. However, if $z=0$, then the expression in \eqref{eq:oneK3two} evaluates to $5/27>0.121423$. 

So, we focus on the other critical points; i.e. the roots of the function $h(z) := 9z^4 + 10z^3+24z^2-6z-28$. Note that $h(0)<0$ and 
\[\frac{d^2h}{dz^2} = \frac{119}{3} + \frac{1}{3} (5 + 18 z)^2\]
which is positive for all $z$. Thus, there is at most one root of $h$, say $z_0$, in the interval $[0,1]$. The exact value of this root is rather complicated, so we will not attempt to write it down, but we can estimate it. Define
\[z_0:=0.908638793,\]
\[z_1:=0.908638794.\]
By plugging these values into $h$, we see $h(z_0)<0$ and $h(z_1)>0$. So, the unique root of $h$ in $[0,1]$ is between $z_0$ and $z_1$. Thus, the minimum of the expression in \eqref{eq:oneK3two} over all $0\leq z\leq 1$ is at least
\[\frac{1}{216}\left(3z_{0}^6 + 4z_{0}^5 + 12z_{0}^4 -4z_{1}^3-28z_{1}^2+40\right) > 0.121423.\]
This completes the proof.
\end{proof}

\begin{rem}
\label{rem:K3K2flag}
By applying the flag algebra method of Razborov~\cite{Razborov07}, one can obtain a lower bound on $c(K_3\sqcup K_2)$ of roughly $0.121449536$ which is slightly better than that of Theorem~\ref{prop:oneK3Lower}; note that this is a raw floating point number from a semi-definite program that was approximately solved by a computer, and so it should be taken with a grain of salt. Thus, while it seems that Lemma~\ref{lem:reduction} may not produce a tight bound in this case, it is within about 0.022\% of the lower bound that one can prove using some of the most powerful modern machinery available. It also has the advantage of being relatively short and human-checkable, which is rarely the case for proofs which use the flag algebra method.
\end{rem}

\section{Reducing the Optimization Problem}
\label{sec:solveOpt}

Our goal in this section is to use a bit of basic calculus to prove Proposition~\ref{prop:reduction}, stated below, which reduces the optimization problem in Lemma~\ref{lem:reduction} to a single-variable problem. In our applications, we will always deal with $k\geq1$ and non-decreasing functions $g$ and $\rho$, and so we will focus on this case.  If $k=1$, then $f^g_{k,\ell}(x,y)$ is linear in the variable $y$ and so, for any $x$, the minimum is achieved at either $y=0$ or $y=g(1+x)-\rho(1+x)$. In other words, when $k=1$, the problem reduces to a single-variable problem trivially; see, e.g., the proof of Proposition~\ref{prop:oneK3Lower} above. Thus, in the following proposition, we only deal with $k>1$. 

\begin{prop}
\label{prop:reduction}
Let $k$ and $\ell$ be real numbers such that $k>1$ and let $g,\rho:[0,2]\to[0,\infty)$ be non-decreasing functions such that $\rho\leq g$ and $\rho(2)>0$. For $0<c\leq 2g(1)^k$, if, for every $0\leq x\leq 1$, either
\begin{equation}
\label{eq:ygood}
\frac{(1+x)^{\ell}(1-x)^{\ell}\left(g(1+x)+g(1-x)\right)^k}{\left((1+x)^{\frac{\ell}{k-1}}+ (1-x)^{\frac{\ell}{k-1}}\right)^{k-1}}\geq c.
\end{equation}
or both of the following hold
\begin{equation}
\label{eq:ybad}
\frac{g(1+x)(1+x)^{\frac{\ell}{k-1}}-g(1-x)(1-x)^{\frac{\ell}{k-1}}}{(1+x)^{\frac{\ell}{k-1}}+ (1-x)^{\frac{\ell}{k-1}}}\geq g(1+x)-\rho(1+x),
\end{equation}
\begin{equation}
\label{eq:ybad2}
(1+x)^{\ell}\rho(1+x)^k+(1-x)^{\ell}\left(g(1-x)+g(1+x)-\rho(1+x)\right)^k\geq c,
\end{equation}
then $f^g_{k,\ell}(x,y)\geq c$ for all $-1\leq x\leq 1$ and $0\leq y\leq g(1+x)-\rho(1+x)$.
\end{prop}

In order to prove this proposition, we require the following basic inequality.

\begin{lem}
\label{lem:Bernoulli}
For any real numbers $a,b$ and $k$ such that $a\geq0$, $a+b\geq0$ and $k>1$, 
\[(a+b)^k\geq a^k+kba^{k-1}.\]
\end{lem}

\begin{proof}
If $a=0$, then the inequality $a+b\geq0$ simply translates to $b\geq0$. In this case, the left side is equal to $b^k\geq0$ and the right size is zero, and so the inequality holds. Now, if $a>0$, then, by Bernoulli's Inequality (i.e. the fact that $(1+z)^r\geq 1+rz$ provided that $1+z\geq0$ and $r\geq1$), 
\[(a+b)^k=a^k\left(1+\frac{b}{a}\right)^k\geq a^k\left(1+\frac{kb}{a}\right)=a^k+kba^{k-1}.\]
This completes the proof.
\end{proof}

\begin{proof}[Proof of Proposition~\ref{prop:reduction}]
We start by showing that range of $x$ in the optimization problem in Lemma~\ref{lem:reduction} can be reduced to $0<x\leq 1$. Fix $-1\leq x\leq 0$. Applying Lemma~\ref{lem:Bernoulli} to both terms of $f^g_{k,\ell}(x,y)$, we get
\[f^g_{k,\ell}(x,y)=(1+x)^{\ell}\left(g(1+x)-y\right)^k+(1-x)^{\ell}\left(g(1-x)+y\right)^k\]
\[\geq (1+x)^{\ell}g(1+x)^k +(1-x)^{\ell}g(1-x)^k -yk\left(g(1+x)^{k-1}(1+x)^{\ell}-g(1-x)^{k-1}(1-x)^{m(k-1)+\ell}\right).\]
Now, using the facts that $y\geq0$, $g$ is non-decreasing and $k>1$, we see that, if $x\leq0$, then the above expression is at least
\[(1+x)^{\ell}g(1+x)^k +(1-x)^{\ell}g(1-x)^k.\]
However, this is precisely equal to $f^g_{k,\ell}(-x,0)$. Therefore, under the constraint $y\geq0$, the minimum of $f^g_{k,\ell}(x,y)$ must be attained at a point $(x,y)$ such that $x\geq0$. If $x=0$, then the above expression evaluates to $2g(1)^k$ which is at least $c$ by assumption. So, we can assume that $0<x\leq1$.

From here forward, fix $0< x\leq 1$. We find and analyze the critical points with respect to $y$. Observe that
\begin{equation}\label{eq:partialy}\frac{\partial f^g_{k,\ell}}{\partial y}(x,y)= -k(1+x)^\ell\left(g(1+x)-y\right)^{k-1}+k(1-x)^\ell\left(g(1-x)+y\right)^{k-1}.\end{equation}
Thus, for any choice of $x$, the partial derivative of $f^g_{k,\ell}(x,y)$ with respect to $y$ is zero at $y=y_0(x)$ where
\[y_0(x):=\frac{(1+x)^{\frac{\ell}{k-1}}g(1+x)- (1-x)^{\frac{\ell}{k-1}}g(1-x)}{(1+x)^{\frac{\ell}{k-1}}+ (1-x)^{\frac{\ell}{k-1}}}.\]
There may also be a second such point $y_1(x)$ where
\[y_1(x):=\frac{(1+x)^{\frac{\ell}{k-1}}g(1+x)+ (1-x)^{\frac{\ell}{k-1}}g(1-x)}{(1+x)^{\frac{\ell}{k-1}}- (1-x)^{\frac{\ell}{k-1}}}.\]
However, 
\[y_1(x)=\frac{(1+x)^{\frac{\ell}{k-1}}g(1+x)+ (1-x)^{\frac{\ell}{k-1}}g(1-x)}{(1+x)^{\frac{\ell}{k-1}}- (1-x)^{\frac{\ell}{k-1}}} \geq \frac{(1+x)^{\frac{\ell}{k-1}}g(1+x)+ (1-x)^{\frac{\ell}{k-1}}g(1-x)}{(1+x)^{\frac{\ell}{k-1}}}\]
\[\geq g(1+x)\geq g(1+x)-\rho(1+x).\]
So, $y_1(x)\geq g(1+x)-\rho(1+x)$. We claim that this inequality is strict; indeed, if $x<1$, then the first two inequalities above are both strict and, if $x=1$, then we have $g(1+x)>g(1+x)-\rho(1+x)$ because $\rho(2)>0$. Thus, when $y=y_1(x)$, the constraint $y\leq g(1+x)-\rho(1+x)$ is violated, and so we can ignore the possibility that $y=y_1(x)$ in what follows. 

From this, we can conclude that, for fixed $0< x\leq 1$, the minimum of the function $f^g_{k,\ell}(x,y)$ on $0\leq y\leq g(1+x)-\rho(1+x)$ is achieved at either $y=0$, $y=y_0(x)$ or $y=g(1+x)-\rho(1+x)$. However, plugging $y=0$ into \eqref{eq:partialy} and using the fact that $x>0$ and $g$ is non-decreasing reveals that $f^g_{k,\ell}(x,y)$ is decreasing with respect to $y$ at the point $y=0$. Thus, the only possibilities to consider are $y=y_0(x)$ or $y=g(1+x)-\rho(1+x)$. To determine which of the two is the minimum, we use the second derivative; we have
\[\frac{\partial^2 f^g_{k,\ell}}{\partial y^2}(x,y)= k(k-1)(1+x)^\ell\left(g(1+x)-y\right)^{k-2}+k(k-1)(1-x)^\ell\left(g(1-x)+y\right)^{k-2}\]
which is clearly positive for all $y$ in the range $0\leq y\leq g(1+x)$. Therefore, we conclude that the minimum is attained at $y=y_0(x)$ whenever $y_0(x)< g(1+x)-\rho(1+x)$ and it is attained at $y=g(1+x)-\rho(x)$ otherwise. The inequality in \eqref{eq:ygood} simply translates to $f(x,y_0(x))\geq c$, whereas \eqref{eq:ybad} and \eqref{eq:ybad2} are equivalent to $y_0(x)\geq g(1+x)-\rho(1+x)$ and $f(x,g(1+x)-\rho(1+x))\geq c$, respectively. This completes the proof.
\end{proof}

Let us briefly remark on the way that the interplay between the two conditions in Proposition~\ref{prop:reduction} tends to work in practice. For fixed $g$ and $\ell$, when $k$ is large, one can often get that \eqref{eq:ygood} holds for all $x$ sufficiently close to zero, i.e., in an interval of the form $[0,x_0]$. Conversely, \eqref{eq:ybad} and \eqref{eq:ybad2} tend to hold when $x$ is close to one, i.e., in an interval of the form $[x_1,1]$; this explains the way that we have labelled of the three inequalities of Proposition~\ref{prop:reduction}. To prove that that $f^g_{k,\ell}(x,y)\geq c$ for all $x$ and $y$ satisfying the constraints, it is therefore sufficient to show that $x_1\leq x_0$.

\section{Triangles, Edges and Triangle-Trees}
\label{sec:K3}

Our goal in this section is to prove the following two theorems, which imply Theorems~\ref{th:DE(K3cupK3)} and~\ref{th:5/3}. After proving them, we will provide several applications to concrete families of graphs built up from gluing together triangles and edges in a ``tree-like'' way.

\begin{thm}
\label{th:correlated2}
Let $0\leq \ell\leq 2$ be an integer. If $H$ is $(K_3,2,\ell)$-correlated, then $H$ is common. 
\end{thm}

\begin{thm}
\label{th:correlated5/3}
Let $k\geq3$ and $0\leq \ell\leq 5k/3$ be integers. If $H$ is $(K_3,k,\ell)$-correlated, then $H$ is common. 
\end{thm}

We deduce Theorems~\ref{th:DE(K3cupK3)} and~\ref{th:5/3} from these two results, after which we will turn our attention to their proofs. 

\begin{proof}[Proof of Theorem~\ref{th:DE(K3cupK3)}]
By Observation~\ref{obs:union}, for $0\leq \ell\leq 2$, the graph $H:=(2\cdot K_3)\sqcup (\ell\cdot K_2)$ is $(K_3,2,\ell)$-correlated. Thus, it is common by Theorem~\ref{th:correlated2}.
\end{proof}

\begin{proof}[Proof of Theorem~\ref{th:5/3}]
Let $k\geq3$ and $0\leq \ell\leq \lfloor 5k/3\rfloor$ be integers. By Observation~\ref{obs:union}, the graph $H:=(k\cdot K_3)\sqcup (\ell\cdot K_2)$ is $(K_3,k,\ell)$-correlated. Thus, it is common by Theorem~\ref{th:correlated5/3}.
\end{proof}

\begin{proof}[Proof of Theorem~\ref{th:correlated2}]
Let $0\leq \ell\leq2$ and let $H$ be $(K_3,2,\ell)$-correlated. By Theorem~\ref{th:GoodmanRamsey} and Observation~\ref{obs:z^m}, $K_3$ is $g$-bounded where $g(z)=z^m$ for $z\in [0,2]$. Let $\rho:[0,2]\to [0,\infty)$ be defined by $\rho(z)=\max\{0,16(z-1)/3\}$ for $0\leq x\leq 2$ and observe that $\rho\leq g$ and $\rho\leq \rho_{K_3}$  by Corollary~\ref{cor:Fisher}. By Lemma~\ref{lem:reduction}, $2^{6+\ell}c(H)$ is at least the minimum of $f^g_{2,\ell}(x,y)$ over all $-1\leq x\leq 1$ and $0\leq y\leq (1+x)^m-\rho(1+x)$. We will be done if we can show that this minimum is at least $2$. 

By Proposition~\ref{prop:reduction}, it suffices to show that every $x\in [0,1]$ satisfies \eqref{eq:ygood} or \eqref{eq:ybad} and \eqref{eq:ybad2} for $c=2$ and the functions $g$ and $\rho$ defined above. By Lemma~\ref{lem:simplifyygood} with $\ell_0=2$, we see that, to prove \eqref{eq:ygood} for a particular value of $x$, it suffices to show that \eqref{eq:ygoodweaker} holds. In this case, this translates to
\begin{equation}\label{eq:ygood2K3}((1+x)^3+(1-x)^3)^2\geq \frac{2((1+x)^2+(1-x)^2)}{(1+x)^2(1-x)^2}.\end{equation}
The above inequality is equivalent to $h(x)\geq 2$ where $h:[0,1)\to \mathbb{R}$ is defined by 
\[h(x):=\frac{(1 - x)^2 (1 + x)^2 ((1 - x)^3 + (1 + x)^3)^2}{(1 - x)^2 + (1 + x)^2}.\]
Note that $h(0)=2$. Also,
\[\frac{dh}{dx}(x) = \frac{4 x (x - 1) (x + 1) (1 + 3 x^2) (-3 + 10 x^2 + 9 x^4)}{(1 + x^2)^2}\]
which is non-negative for all $0\leq x\leq \frac{1}{3}\sqrt{-5 + 2 \sqrt{13}}\approx 0.495$. Thus, \eqref{eq:ygood2K3} holds for all $0\leq x\leq 0.49$. 

So, it suffices to show that \eqref{eq:ybad} and \eqref{eq:ybad2} both hold whenever $0.49\leq x\leq 1$. By Lemma~\ref{lem:simplifyybad}, to prove that a particular $x$ satisfies \eqref{eq:ybad}, it is enough to show that \eqref{eq:ybadweaker} holds. That is, 
\[2\rho(1+x)\geq (1+x)^3+(1-x)^3.\]
Plugging in $\rho(1+x)=16x/3$, the inequality becomes $32x/3 \geq (1+x)^3+(1-x)^3$. This holds for all $x_1\leq x\leq 1$ where $x_1=\frac{1}{9}(8-\sqrt{37})\approx 0.21 < 0.49$. Finally, we analyze \eqref{eq:ybad2}. It simplifies to
\[(1+x)^{\ell}(16x/3)^2+(1-x)^{\ell}\left((1-x)^3+(1+x)^3-16x/3\right)^2\geq 2.\]
For $x\geq0.49$ and $\ell\geq0$, we have
\[(1+x)^{\ell}(16x/3)^2+(1-x)^{\ell}\left((1-x)^3+(1+x)^3-16x/3\right)^2 \geq (16x/3)^2>6.8 >2\]
and so \eqref{eq:ybad2} holds for all $x\geq 0.49$. The result follows. 
\end{proof}

Next, we prove the following lemma from which Theorem~\ref{th:correlated5/3} will be derived. The proof is similar to that of  Theorem~\ref{th:correlated2}, except that we sometimes need to deal with the more complicated inequalities \eqref{eq:ygood} and \eqref{eq:ybad} themselves as opposed to the sufficient conditions from Appendix~\ref{appendix}.

\begin{lem}
\label{lem:k=3}
For any integer $r$ such that $0\leq r\leq 15$ and any graphon $W$,
\begin{equation}\label{eq:k=3}2^{9+r/3}\left(t(K_2,W)^{r/3}t(K_3,W)^3 + t(K_2,1-W)^{r/3}t(K_3,1-W)^3\right)\geq2.\end{equation}
\end{lem}

\begin{proof}
We apply Lemma~\ref{lem:reduction} and Proposition~\ref{prop:reduction} with $J=K_3$, $\rho(z)=16(z-1)/3$, $g(z)=z^3$, $k=3$ and $\ell=r/3$. The inequality in \eqref{eq:ygood} becomes
\begin{equation}
\label{eq:ygood3K3}
\frac{(1+x)^{r/3}(1-x)^{r/3}\left((1+x)^3+(1-x)^3\right)^3}{\left((1+x)^{\frac{r}{6}}+(1-x)^{\frac{r}{6}}\right)^{2}}\geq 2.
\end{equation}
By Lemma~\ref{lem:simplifyygood2}, to prove \eqref{eq:ygood3K3}, it suffices to show that 
\begin{equation}
\label{eq:ygoodweaker3K3}
(1+x)^{r/3}(1-x)^{r/3}\left(\frac{(1+x)^3+(1-x)^3}{2}\right)^{3-\frac{r}{9}}\geq 1.
\end{equation}
For each integer $r$ such that $0\leq r\leq 13$, one can verify numerically that \eqref{eq:ygoodweaker3K3} is satisfied for all $x$ in the range $0\leq x\leq x_{0,r}$ where
\[x_{0,0}=1, \quad x_{0,1}\geq 0.99, \quad x_{0,2}\geq 0.99,\quad  x_{0,3}\geq 0.98,\quad  x_{0,4}\geq 0.95,\quad x_{0,5}\geq 0.91,\]
\[x_{0,6}\geq 0.86, \quad x_{0,7}\geq 0.79,\quad x_{0,8}\geq 0.72,\quad x_{0,9}\geq 0.64,\quad x_{0,10}\geq 0.55,\quad x_{0,11}\geq 0.46,\]
\[ x_{0,12}\geq 0.34,\quad x_{0,13}\geq 0.19.\]
Also, for $r\in \{14,15\}$, one can verify numerically that \eqref{eq:ygood3K3} is satisfied whenever $x$ is in the range $0\leq x\leq x_{0,r}$ where $x_{0,r}\geq 0.14$; the worst case is $r=15$.

So, to complete the proof, we need to show that, for each $r$, \eqref{eq:ybad} and \eqref{eq:ybad2} hold for all $x_{0,r}\leq x\leq1$. Plugging in the values of $k,\ell,\rho,g$, these inequalities translate to 
\begin{equation}\label{eq:ybadK3}\frac{(1+x)^3(1+x)^{\frac{r}{6}}-(1-x)^3(1-x)^{\frac{r}{6}}}{(1+x)^{\frac{r}{6}}+ (1-x)^{\frac{r}{6}}}\geq (1+x)^3-16x/3\end{equation}
and
\begin{equation}\label{eq:ybad2K3}(1+x)^{r/3}(16x/3)^3+(1-x)^{r/3}\left((1-x)^3+(1+x)^3-16x/3\right)^3\geq 2\end{equation}
respectively. For each integer $0\leq r\leq 15$, \eqref{eq:ybad2K3} holds for all $0\leq x\leq 1$; again, we are only able to verify this numerically. By Lemma~\ref{lem:simplifyybad}, for a given $0\leq x\leq 1$, if
\[32x/3\geq (1+x)^3+(1-x)^3,\]
then \eqref{eq:ybadK3} holds automatically. The above inequality holds for all $0.22\leq x\leq 1$. Thus, since $x_{0,r}>0.22$ for $r\in\{0,\dots,12\}$, the proof is complete in these cases. 

The last thing to do is to verify that \eqref{eq:ybadK3} holds for all $x_{0,r}\leq x\leq 1$ in the cases $r\in\{13,14,15\}$. One can check numerically that the inequality holds for all $0.1397\leq x\leq 1$ for each $r\in\{13,14,15\}$; the worst case is $r=13$. Since $x_{0,r}>0.1397$ for $r\in\{13,14,15\}$, the proof is complete.
\end{proof}

Next, we use Lemma~\ref{lem:k=3} and H\"older's Inequality to prove Theorem~\ref{th:correlated5/3}.

\begin{proof}[Proof of Theorem~\ref{th:correlated5/3}]
Since $H$ is $(K_3,k,\ell)$-correlated, for any graphon $W$, we have
\[t(H,W)+t(H,1-W) \geq t(K_2,W)^\ell t(K_3,W)^k+t(K_2,1-W)^\ell t(K_3,1-W)^k.\]
Define 
\[p:=\frac{k}{3}, \quad q:=\frac{k}{k-3},\]
\[x_1:=t(K_2,W)^{\frac{3\ell}{k}}t(K_3,W)^3,\quad x_2:=t(K_2,1-W)^{\frac{3\ell}{k}}t(K_3,1-W)^3,\]
\[\alpha:=\frac{1}{3}\left(\left\lceil\frac{9\ell}{k}\right\rceil - \frac{9\ell}{k}\right),\quad  r:=\left\lceil\frac{9\ell}{k}\right\rceil,\]
\[y_1:=t(K_2,W)^{\alpha},\quad y_2:=t(K_2,1-W)^{\alpha}.\]
By Lemma~\ref{lem:Holder},
\[\left(x_1^p+x_2^p\right)^{1/p}\left(y_1^q+y_2^q\right)^{1/q}\geq x_1y_1+x_2y_2.\]
In other words,
\[t(K_2,W)^\ell t(K_3,W)^k+t(K_2,1-W)^\ell t(K_3,1-W)^k\]
\[\geq \left(\frac{t(K_2,W)^{\frac{3\ell}{k}+\alpha}t(K_3,W)^3+t(K_2,1-W)^{\frac{3\ell}{k}+\alpha}t(K_3,1-W)^3}{\left(t(K_2,W)^{\alpha q}+t(K_2,1-W)^{\alpha q}\right)^{(k-3)/k}}\right)^{k/3}.\]
By definition of $\alpha$, we have $\frac{3\ell}{k}+\alpha = \frac{r}{3}$. Since $0\leq \ell\leq 5k/3$, we have that $r$ is an integer satisfying $0\leq r\leq 15$. So, by Lemma~\ref{lem:k=3}, we get that the right side of the above inequality is at least
\[\frac{2^{(-8-r/3)(k/3)}}{\left(t(K_2,W)^{\alpha q}+t(K_2,1-W)^{\alpha q}\right)^{(k-3)/3}}.\]
Note that 
\[\alpha q = \frac{1}{3}\left(\left\lceil\frac{9\ell}{k}\right\rceil - \frac{9\ell}{k}\right)\left(\frac{k}{k-3}\right)\leq \frac{1}{3}\left(\frac{k-1}{k}\right)\left(\frac{k}{k-3}\right) = \frac{1}{3}\left(\frac{k-1}{k-3}\right)\leq 1\]
where the last inequality uses that $k\geq4$. Thus, since $t(K_2,W)+t(K_2,1-W)=1$, we have that $t(K_2,W)^{\alpha q}+t(K_2,1-W)^{\alpha q}\leq (1/2)^{\alpha q} + (1/2)^{\alpha q}$ by Jensen's Inequality. So, we get
\[\frac{2^{(-8-r/3)(k/3)}}{\left(t(K_2,W)^{\alpha q}+t(K_2,1-W)^{\alpha q}\right)^{(k-3)/3}}\geq \frac{2^{-\frac{8k}{3}-\frac{kr}{9}}}{\left((1/2)^{\alpha q} + (1/2)^{\alpha q}\right)^{(k-3)/3}}= \frac{2^{-\frac{8k}{3}-\frac{kr}{9}}}{2^{(k-3)/3}(1/2)^{\alpha \left(k/3\right)}}\]
\[= \frac{2^{-\frac{8k}{3}-\frac{kr}{9}}}{2^{k/3-1}2^{-\alpha \left(k/3\right)}}= \frac{2^{1-3k-\frac{kr}{9}}}{2^{-\frac{1}{3}\left(r- \frac{9\ell}{k}\right) \left(k/3\right)}}= \frac{2^{1-3k-\frac{kr}{9}}}{2^{-\frac{kr}{9}+\ell}}=2\left(1/2\right)^{3k+\ell}=2\left(1/2\right)^{e(H)}\]
where the last equality follows from \eqref{eq:eCondition} and the fact that $H$ is $(K_3,k,\ell)$-correlated. Therefore, $H$ is common and we are done. 
\end{proof}

To close this section, we provide some concrete families of $(K_3,k,\ell)$-correlated graphs. Following~\cite{GrzesikLeeLidickyVolec22}, the class $\mathcal{T}$ of \emph{triangle-trees} is defined as follows. We have $K_3\in\mathcal{T}$ and, if $H\in\mathcal{T}$, then any graph obtained from $H$ by either adding two vertices $u,v$ to $H$ and edges $uv,uw,vw$ for some $w\in V(H)$, or adding one vertex $x$ to $H$ and edges $xy$ and $xz$ for some $yz\in E(H)$, is also in $\mathcal{T}$. In other words, the graphs in $\mathcal{T}$ are built up by gluing triangles together on vertices or edges in a tree-like manner. Our results in this section apply to a slightly more general class of graphs (see Example~\ref{ex:TT}), which is a special case of the class considered in~\cite[Section~6]{Behague+Second}. For any set $X$, let $2^X$ denote the power set of $X$.

\begin{defn}
Given a graph $F$, define an \emph{$F$-tree} to be a pair $(T,\varphi)$ such that $T$ is a tree and $\varphi:V(T)\cup E(T)\to 2^{V(F)}$ satisfies $\varphi(st)\subsetneq \varphi(s)\cap \varphi(t)$ for all $st\in E(T)$. 
\end{defn}

\begin{defn}
Let $F$ be a graph and $(T,\varphi)$ be an $F$-tree. Define $H(T,\varphi)$ to be the graph constructed by taking pairwise disjoint graphs $J_t$ for $t\in V(T)$, where $J_t$ is isomorphic to $F[\varphi(t)]$, and then, for each $st\in E(T)$, identifying the vertex corresponding to $w$ in the $J_s$ with the vertex corresponding to $w$ in $J_t$ for all $w\in \varphi(st)$.
\end{defn}

\begin{ex}
\label{ex:union}
Let $F$ be a graph, let $T$ be any tree with $k$ vertices and suppose that $\varphi$ maps every vertex of $T$ to $V(F)$ and every edge of $T$ to $\emptyset$. Then $H(T,\varphi)$ is nothing more than the graph $k\cdot F$.
\end{ex}

\begin{ex}
\label{ex:TT}
Let $(T,\varphi)$ be a $K_3$-tree such that $|\varphi(t)|=3$ for all $t\in V(T)$ and $1\leq|\varphi(e)|\leq 2$ for all $e\in E(T)$. Then $H(T,\varphi)$ is a triangle-tree. Moreover, every triangle-tree can be represented as $H(T,\varphi)$ for some such $T$ and $\varphi$. 
\end{ex}

\begin{defn}
Given a $K_3$-tree $(T,\varphi)$ and $1\leq j\leq 3$, define
\[v_j(T,\varphi):=|\{t\in V(T): |\varphi(t)|=j\}|.\]
Also, for $0\leq j\leq 2$, define
\[e_j(T,\varphi):=|\{e\in E(T): |\varphi(e)|=j\}|.\]
\end{defn}

The statement of the following lemma is technically slightly more general than~\cite[Lemma~3.6]{GrzesikLeeLidickyVolec22}. However, the argument in the proof of~\cite[Lemma~3.6]{GrzesikLeeLidickyVolec22} applies, with virtually no modifications, to prove it; for this reason, we attribute it to~\cite{GrzesikLeeLidickyVolec22}. A more general (but also more technical) statement for odd cycles which implies this lemma can be found in~\cite[Corollary~6.32]{Behague+Second}.

\begin{lem}[Grzesik, Lee, Lidick\'y and Volec~{\cite[Lemma~3.6]{GrzesikLeeLidickyVolec22}}]
\label{lem:tt}
Let $(T,\varphi)$ be a $K_3$-tree, let $k=v_3(T,\varphi)$, let $\gamma=e_2(T,\varphi)-v_2(T,\varphi)\geq0$. If $H=H(T,\varphi)$, then $H$ is $(K_3,k,-\gamma)$-correlated.
\end{lem}

We derive two consequences of Theorems~\ref{th:correlated2} and~\ref{th:correlated5/3}. Note that, since $\ell\cdot K_2$ is Sidorenko for all $\ell$, these results imply lower bounds on the number of disjoint edges that can be added to $H$ while maintaining the property that it is common whenever $(T,\varphi)$ is a $K_3$-tree and $H=H(T,\varphi)$.

\begin{cor}
\label{cor:2TT}
Let $(T,\varphi)$ be a $K_3$-tree such that $v_3(T,\varphi)=2$ and $e_2(T,\varphi)\geq v_2(T,\varphi)$. If $H=H(T,\varphi)$ and $F$ is a Sidorenko graph such that
\[e(F)\leq 2+e_2(T,\varphi)- v_2(T,\varphi),\]
then $H\sqcup F$ is common.
\end{cor}

\begin{proof}
Define $\ell=e(F) -e_2(T,\varphi)+ v_2(T,\varphi)$.  By Observation~\ref{obs:union}, Lemma~\ref{lem:tt} and the fact that $F$ is Sidorenko, we have that, for any graphon $W$,
\[t(H\sqcup F,W) = t(H,W)t(F,W) \geq t(K_3,W)^2t(K_2,W)^{-e_2(T,\varphi)+ v_2(T,\varphi)} t(K_2,W)^{e(F)}\]
\[= t(K_3,W)^2t(K_2,W)^\ell.\]
Also, since $H$ is $(K_3,2,-e_2(T,\varphi)+v_2(T,\varphi))$-correlated by Lemma~\ref{lem:tt}, we have
\[e(H\sqcup F) = e(H)+e(F) = 2e(K_3)-e_2(T,\varphi)+v_2(T,\varphi) +e(F)\]
\[=2e(K_3) +\ell.\]
Thus, $H\sqcup F$ is $(K_3,2,\ell)$-correlated, where $\ell\leq2$ by hypothesis. The result now follows from Theorem~\ref{th:correlated2}.
\end{proof}

\begin{cor}
\label{cor:5/3TT}
Let $(T,\varphi)$ be a $K_3$-tree such that $v_3(T,\varphi)=k\geq3$ and $e_2(T,\varphi)\geq v_2(T,\varphi)$. If $H=H(T,\varphi)$ and $F$ is a Sidorenko graph such that
\[e(F)\leq 5k/3+e_2(T,\varphi)- v_2(T,\varphi),\]
then $H\sqcup F$ is common.
\end{cor}

\begin{proof}
The proof is completely analogous to that of Corollary~\ref{cor:2TT}.
\end{proof}

\section{Beating the Random Construction}
\label{sec:uncommon}

In this section, we turn our attention to obtaining upper bounds on $c(H)$ for various graphs $H$. The graphons that we will use are all of the same general form. For $n\geq1$, let $\triangle_n$ be the set of all vectors $\vec{z}$ of length $n$ with non-negative entries that sum to one. Given $\vec{z}\in \triangle_n$ and an $n\times n$ symmetric matrix $A$  with entries in $[0,1]$, let $W_{\vec{z},A}$ be defined as follows. First, divide $[0,1]$ into $n$ intervals $I_1,\dots,I_n$ such that the measure of $I_i$ is equal to $\vec{z}_i$. Next, for each $1\leq i,j\leq n$, define $W_{\vec{z},A}$ to be equal to $A_{i,j}$ for all $(x,y)\in I_i\times I_j$. It is easily observed that, for any graph $H$,
\begin{equation}\label{eq:partition function}
t(H,W_{\vec{z},A})=\sum_{f:V(H)\to [n]}\prod_{v\in V(H)}\vec{z}_{f(v)}\prod_{uv\in E(H)}A_{f(u),f(v)}.\end{equation}

\begin{rem}
In the statistical physics literature, the pair $(\vec{z},A)$ is often referred to as a ``spin system'' and $t(H,W_{\vec{z},A})$ is the ``partition function'' of that spin system. 
\end{rem} 

The constructions used to prove Theorem~\ref{th:paw}, the upper bound in Theorem~\ref{th:oneK3} and to show that Theorem~\ref{th:DE(K3cupK3)} is best possible are similar to one another; we present them next. 

\begin{prop}
\label{prop:DE(K_3cupK_3)upper}
The graph $(2\cdot K_3) \sqcup (3\cdot K_2)$ is uncommon
\end{prop}

\begin{proof}
We prove that the graph $H=(2\cdot K_3) \sqcup (3\cdot K_2)$ is uncommon. For $z\in [0,1/2]$ and $y\in [0,1]$, we define $W_{z,y}:=W_{\vec{z},A}$ where $\vec{z}=\left(1-2z, z,z \right)\in \triangle_3$ and $A$ is the symmetric $3\times 3$ matrix in which $A(1,2)=A(1,3)=1$, $A(2,3)=y$ and $A(i,i)=0$ for $1\leq i\leq 3$. Using \eqref{eq:partition function}, one can compute
\[t(K_2,W_{z,y}) = 4z(1-2z) + 2z^2y,\]
\[t(K_2,1-W_{z,y}) = (1-2z)^2 + 2z^2(2-y),\]
\[t(K_3,W_{z,y}) = 6z^2(1-2z)y \]
and
\[t(K_3,1-W_{z,y}) = (1-2z)^3 + z^3(2+6(1-y)^2).\]
Therefore, by Observation~\ref{obs:union}, $t(H,W_{z,y})+t(H,1-W_{z,y})=h(z,y)$ where
 \begin{align*}h(z,y)&=(4z(1-2z) + 2z^2y)^3 \cdot (6z^2(1-2z)y)^2\\& + ((1-2z)^2 + 2z^2(2-y))^3 \cdot ((1-2z)^3 + z^3(2+6(1-y)^2))^2.\end{align*}
Thus, $c(H)$ is at most the minimum of $h(z,y)$ over all $0\leq z\leq 1/2$ and $0\leq y\leq 1$. Setting $z=0.28$ and $y=0.42$ yields $h(z,y)=0.00390226 < 2\cdot\left(\frac{1}{2}\right)^9$, which completes the proof.
\end{proof}

\begin{proof}[Proof of Theorem~\ref{th:oneK3}]
See Proposition~\ref{prop:oneK3Lower} for the lower bound. For the upper bound, we need to prove $c(K_3\sqcup K_2)<0.12145$. We let $W_{z,y}$ be the graphon as in the proof of Theorem~\ref{th:DE(K3cupK3)} above. By Observation~\ref{obs:union}, $t(K_3\sqcup K_2,W_{z,y})+t(K_3\sqcup K_2,1-W_{z,y})=h(z,y)$ where

 \begin{align*}h(z,y)&=(4z(1-2z) + 2z^2y) \cdot (6z^2(1-2z)y)\\& + ((1-2z)^2 + 2z^2(2-y)) \cdot ((1-2z)^3 + z^3(2+6(1-y)^2)).\end{align*}
Thus, $c(K_3\sqcup K_2)$ is at most the minimum of $h(z,y)$ over all $0\leq z,y\leq 1$. Setting $z=0.263661$ and $y=0.2177$ yields $h(z,y)=0.12145$, which completes the proof.
\end{proof}

\begin{proof}[Proof of Theorem~\ref{th:paw}]
Let $P$ be the paw graph and let $W_{z,y}$ be the graphon as in the proof of Theorems~\ref{th:DE(K3cupK3)} and~\ref{th:oneK3}. Using \eqref{eq:partition function}, we get that 

\[t(P,W_{z,y})=2z^2(1-2z)y(2z+2(1-2z)+2zy) = 4z^2(1-z)y(1-z(1-y))
\]
and 
\[t(P,1-W_{z,y})=(1-2z)^4 + 2z^4(2-y) + 6z^4 \cdot (1-y)^2 \cdot (2-y) 
= (1-2z)^4 +z^4(2-y)(2+6(1-y)^2).\]

Thus, $t(P,W_{z,y})+t(P,1-W_{z,y})$ is equal to $h(z,y)$ where
\begin{align*}
h(z,y) &= 4z^2(1-z)y(1-z(1-y))+(1-2z)^4 +z^4(2-y)(2+6(1-y)^2).
\end{align*}
Therefore, $c(P)<h(0.266491,0.2187477)=0.121415$, which completes the proof. 
\end{proof}

Recall that, by Theorem~\ref{th:5/3}, the graph obtained from $K_3$ by adding five disjoint edges is common. Our next example shows that a very closely related graph, obtained from $3\cdot K_3$ by adding three pendant edges in different components and two disjoint edges, is uncommon

\begin{prop}
\label{prop:pendant}
$(3\cdot P)\sqcup (2\cdot K_2)$ is uncommon. 
\end{prop}

\begin{proof}
Let $H=(3\cdot P)\sqcup (2\cdot K_2)$. Once again, we use the graphon $W_{z,y}$ from the previous three proofs. This time, we set $z=0.429919$ and $y=0.43222$. We get
\[t(K_2,W_{z,y})=0.560411,\]
\[t(K_2,1-W_{z,y})=0.439589,\]
%\[t(K_3,W_{z,y})=0.0905849,\]
%\[t(K_3,1-W_{z,y})=0.170575,\]
\[t(P,W_{z,y})=0.0506164,\]
\[t(P,1-W_{z,y})=0.074879.\]
Thus, $t(H,W_{z,y})+t(H,1-W_{z,y}) < 0.000121856 < 2(1/2)^{14}$ and the result follows. 
\end{proof}

Next, we provide an example which demonstrates that not all common graphs are strongly common; the question of whether such graphs exist was raised in~\cite{Behague+Second}. Note that the fact that $K_3\sqcup K_3$ is common can be easily deduced from Theorem~\ref{th:GoodmanRamsey}.

\begin{prop}
\label{prop:K3K3}
$K_3\sqcup K_3$ is not strongly common. 
\end{prop}

\begin{proof}
Let $\vec{z}=(1/2,1/2)$, let $A$ be the $2\times 2$ symmetric matrix such that $A(1,1)=A(2,2)=1/3$ and $A(1,2)=1$ and define $W=W_{\vec{z},A}$. Then, by \eqref{eq:partition function},
\[t(K_2,W) = 2/3,\]
\[t(K_2,1-W)=1/3,\]
\[t(K_3,W) = \left(\frac{1}{3}\right)^3\left(\frac{1}{2}\right)^3 + 6\left(\frac{1}{3}\right)\left(\frac{1}{2}\right)^3 = \frac{55}{216}\]
and
\[t(K_3,1-W) = \left(\frac{2}{3}\right)^3\left(\frac{1}{2}\right)^3=\frac{1}{27}.\]
Therefore, 
\[t(K_3\sqcup K_3,W)+t(K_3\sqcup K_3,1-W)=\left(\frac{55}{216}\right)^2 + \left(\frac{1}{27}\right)^2 = \frac{3089}{46656} < \frac{65}{729}\]
\[=t(K_2,W)^6+t(K_2,1-W)^6\]
and so $K_3\sqcup K_3$ is not strongly common. 
\end{proof}

The first common graph with chromatic number four which was discovered is  the $5$-wheel~\cite{Hatami+12}; i.e. the graph $W_5$ obtained from a cycle of length five by adding a vertex joined to everything on the cycle. The following proposition implies, for example, that $W_5$ is not strongly common. 

\begin{prop}
\label{prop:W5}
Let $H$ be a graph with $m$ components and chromatic number $k$. If
\[\left(\frac{1}{k-1}\right)^{v(H)-m} < \left(\frac{k-2}{k-1}\right)^{e(H)} + \left(\frac{1}{k-1}\right)^{e(H)},\]
then $H$ is not strongly common. 
\end{prop}

\begin{proof}
Let $W=W_{K_{k-1}}$; i.e. the graphon corresponding to the complete graph $K_{k-1}$. Then, $t(K_2,W)=(k-2)/(k-1)$ and $t(K_2,1-W)=1/(k-1)$. Since $H$ has chromatic number $k$,
\[t(H,W)=0.\]
Also,
\[t(H,1-W)=\left(\frac{1}{k-1}\right)^{v(H)-m}.\]
So, by the hypothesis of the proposition, $H$ is not strongly common. 
\end{proof}

\begin{cor}
\label{cor:W5}
$W_5$ is not strongly common.
\end{cor}

\begin{proof}
The chromatic number of $W_5$ is four and
\[\left(\frac{1}{4-1}\right)^{6-1} = \frac{1}{243} <\frac{1025}{59049} =  \left(\frac{4-2}{4-1}\right)^{10} + \left(\frac{1}{4-1}\right)^{10}.\]
Thus, the result follows by Proposition~\ref{prop:W5}.
\end{proof}

Finally, we prove Theorem~\ref{th:uncommon}.

\begin{proof}[Proof of Theorem~\ref{th:uncommon}]
Let $k\geq1$ and $\ell=\lceil1.9665k\rceil$. We show that $H=(k\cdot K_3)\sqcup (\ell\cdot K_2)$ is uncommon. Define $\alpha=\ell/k$ and note that $1.9665\leq\alpha\leq2$. Let
\[p:=1-2^{-1/(3 + \alpha)}.\]
We let $W$ be the graphon $W_{\vec{z},A}$ where $\vec{z}=(1/2,1/2)$ and $A$ is a $2\times 2$ matrix whose diagonal entries are $p$ and off-diagonal entries are $1$. We have
\[t(K_2,W)=\frac{1}{2}(1+p),\]
\[t(K_2,1-W)=\frac{1}{2}(1-p),\]
\[t(K_3,W)=\frac{1}{4}p^3+\frac{3}{4}p,\]
\[t(K_3,1-W)=\frac{1}{4}(1-p)^3.\]
Thus, by Observation~\ref{obs:union},
\[t(H,W) = \left[\left(\frac{1}{2}(1+p)\right)^\alpha\left(\frac{1}{4}p^3+\frac{3}{4}p\right)\right]^k=\frac{1}{2^{3k+\ell}}\left[2^{3+\alpha}\left(\frac{1}{2}(1+p)\right)^\alpha\left(\frac{1}{4}p^3+\frac{3}{4}p\right)\right]^k.\]
Plugging in the value of $p$ defined above yields
\[\frac{1}{2^{3k+\ell}}\left[2 \left(1 - 2^{-1/(3 + \alpha)}\right) \left(2 - 2^{-1/(3 + \alpha)}\right)^\alpha \left(3 + \left(2^{-1/(3 + \alpha)}-1\right)^2\right)\right]^k\]
The expression within the square brackets is decreasing in $\alpha$ for $\alpha\in[0,2]$. Therefore, for all $\alpha\geq 1.9665$ it is at most
\[\frac{1}{2^{3k+\ell}}\left[2 \left(1 - 2^{-1/4.9665}\right) \left(2 - 2^{-1/4.9665}\right)^{1.9665} \left(3 + \left(2^{-1/4.9665}-1\right)^2\right)\right]^k\]
\[< \frac{1}{2^{3k+\ell}} (0.9999994)^k.\]

Next, we compute
\[t(H,1-W) = \left[\left(\frac{1}{2}(1-p)\right)^\alpha\frac{1}{4}(1-p)^3\right]^k=\frac{1}{2^{3k+\ell}}\left[2^{3+\alpha}\left(\frac{1}{2}(1-p)\right)^\alpha\frac{1}{4}(1-p)^3\right]^k.\]
When substituting the value of $p$ chosen above into this, the expression inside of the square brackets evaluates to 1. Putting all of this together, we get
\[t(H,W)+t(H,1-W)<\frac{1}{2^{3k+\ell}} (0.9999994)^k + \frac{1}{2^{3k+\ell}} < 2\left(1/2\right)^{e(H)}.\]
This completes the proof. 
\end{proof}

Note that the graphon $W$ used in the proof of Theorem~\ref{th:uncommon} above is vertex-transitive. Thus, if $H$ is a graph obtained from the disjoint union of two graphs $F_1$ and $F_2$ by identifying one vertex of $F_1$ with one vertex of $F_2$, then $t(H,W)=t(F_1,W)t(F_2,W)$ and $t(H,1-W)=t(F_1,1-W)t(F_2,1-W)$ (for the specific graphon $W$ used in the proof). Thus, the same construction can be used to obtain the following.

\begin{thm}
\label{th:ttuncommon}
Let $(T,\varphi)$ be a $K_3$-tree such that $e_2(T,\varphi)=0$ and let $H=H(T,\varphi)$. If
\[v_2(T,\varphi)\geq 1.9665\cdot v_3(T,\varphi),\]
then $H$ is uncommon. 
\end{thm}

A graph $H$ is said to be \emph{positive} if $t(H,W)\geq0$ for every kernel $W$. If $H$ is a graph obtained from two copies of a graph $F$ by gluing them on an independent set, then $H$ is easily seen to be positive, and it is conjectured that all positive graphs arise in this way~\cite{Camarena+16}. Theorem~\ref{th:ttuncommon} provides examples of connected positive graphs that are uncommon, e.g. the graph obtained by taking two triangles and four edges all glued on a single vertex, which answers~\cite[Question~5.1]{KimLee24} in the affirmative (in a strong sense).

The construction in the proof of Theorem~\ref{th:uncommon} can also be used for longer odd cycles, as we show next.

\begin{thm}
For integers $k,r\geq1$, the graph $(k\cdot C_{2r+1})\sqcup (2rk\cdot K_2)$ is uncommon.
\end{thm}

\begin{proof}
Let $H=(k\cdot C_{2r+1})\sqcup (2rk\cdot K_2)$ and let
\[p:=1-2^{-1/(4r+1)}.\]
Let $W$ be the graphon $W_{\vec{z},A}$ where $\vec{z}=(1/2,1/2)$ and $A$ is a $2\times 2$ matrix whose diagonal entries are $p$ and off-diagonal entries are $1$. Note that the eigenvalues of $A$ are $p+1$ and $p-1$ and $J-A$ has eigenvalue $1-p$ with multiplicity two, where $J$ is the $2\times 2$ all-ones matrix. So, using, e.g.,~\cite[Example~5.11]{Lovasz12}, we get
\[t(C_{2r+1},W)=\left(\frac{p+1}{2}\right)^{2r+1} + \left(\frac{p-1}{2}\right)^{2r+1},\]
\[t(C_{2r+1},1-W)=2\left(\frac{1-p}{2}\right)^{2r+1}.\]
Thus, by Observation~\ref{obs:union},
\[t(H,W) = \left[\left(\frac{p+1}{2}\right)^{2r}\left(\left(\frac{p+1}{2}\right)^{2r+1} + \left(\frac{p-1}{2}\right)^{2r+1}\right)\right]^k\]
\[= \frac{1}{2^{(4r+1)k}}\left[(p+1)^{2r}\left((p+1)^{2r+1} + (p-1)^{2r+1}\right)\right]^k = \frac{1}{2^{(4r+1)k}}\left[(p+1)^{4r+1} + (p-1)^{2r+1}(p+1)^{2r}\right]^k.\]
Substituting in the value of $p$ yields
\[\frac{1}{2^{(4r+1)k}}\left[\left(2-2^{-1/(4r+1)}\right)^{4r+1} -2^{-(2r+1)/(4r+1)}\left(2-2^{-1/(4r+1)}\right)^{2r}\right]^k.\]
Our next goal is to show that the expression inside of the square brackets is less than one for all $r\geq 1$. That is, we want to prove that
\[\left(2-2^{-1/(4r+1)}\right)^{4r+1}-2^{-(2r+1)/(4r+1)}\left(2-2^{-1/(4r+1)}\right)^{2r} < 1\]
or, in other words,
\begin{equation}
\label{eq:<1}
\left(2-2^{-1/(4r+1)}\right)^{4r+1} < 1 + 2^{-(2r+1)/(4r+1)}\left(2-2^{-1/(4r+1)}\right)^{2r}.
\end{equation}
The inequality \eqref{eq:<1} can be easily verified by computer for $1\leq r\leq 6$. So, in the rest of the proof of \eqref{eq:<1}, we assume $r\geq7$. By the AM-GM Inequality, the right side of \eqref{eq:<1} can be bounded below as follows:
\[1 + 2^{-(2r+1)/(4r+1)}\left(2-2^{-1/(4r+1)}\right)^{2r} > 2\sqrt{2^{-(2r+1)/(4r+1)}\left(2-2^{-1/(4r+1)}\right)^{2r}}\]
\[=2^{(6r+1)/(8r+2)}\left(2-2^{-1/(4r+1)}\right)^{r}.\]
Thus, to prove \eqref{eq:<1}, it suffices to show that
\begin{equation}\label{eq:anotherThing}2^{(6r+1)/(8r+2)} \geq \left(2-2^{-1/(4r+1)}\right)^{3r+1}\end{equation}
or, in other words, that
\[2^{x - y} + 2^{-x}\geq 2\]
where $x=\frac{1}{4r+1}$ and $y=\frac{1}{2(3r+1)(4r+1)}$. This holds if and only if
\begin{equation}\label{eq:final?}x\geq y + \frac{\log\left(1+\sqrt{1-2^{-y}}\right)}{\log(2)}.\end{equation}
The right side of \eqref{eq:final?} can be bounded above as follows:
\[y+\frac{\log\left(1+\sqrt{1-2^{-y}}\right)}{\log(2)} \leq y+\frac{\sqrt{1-2^{-y}}}{\log(2)}=y+\frac{\sqrt{1-e^{-\log(2)y}}}{\log(2)}\]
\[\leq y+\frac{\sqrt{1-(1-\log(2)y)}}{\log(2)}= y + \sqrt{\frac{y}{\log(2)}}.\]
Thus, to prove \eqref{eq:final?}, it suffices to show that
\[\frac{1}{2(3r+1)(4r+1)} + \sqrt{\frac{1}{\log(8)(3r+1)(4r+1)}} < \frac{1}{4r+1}.\]
This inequality holds for all $r\geq7$. Therefore, we have that $t(H,W)<\frac{1}{2^{e(H)}}$.

Next, we compute
\[t(H,1-W) = \left[\left(\frac{1-p}{2}\right)^{2r}\left(2\left(\frac{1-p}{2}\right)^{2r+1}\right)\right]^k = \frac{1}{2^{(4r+1)k}}\left[2(1-p)^{4r+1}\right]^k.\]
When substituting the value of $p$ chosen above into this, the expression inside of the square brackets evaluates to 1. Putting all of this together, we get
\[t(H,W)+t(H,1-W)<2(1/2)^{(4r+1)k} = 2(1/2)^{e(H)}\]
which completes the proof. 
\end{proof}

Of course, the above theorem can also be extended to graphs built up by gluing together odd cycles and edges on single vertices, analogous to the way that Theorem~\ref{th:ttuncommon} follows from the same proof as Theorem~\ref{th:uncommon}.

\section{Conclusion}

As mentioned in the introduction, one of the most classical examples of an uncommon graph is the paw graph $P$ obtained from $K_3$ by adding a pendant edge. The problem of computing the Ramsey multiplicity constant of $P$ seems to be intimately linked to the analogous problem for $K_3\sqcup K_2$. Given that $2\cdot (K_3\sqcup K_2)$ is common by Theorem~\ref{th:DE(K3cupK3)}, we believe that the same may be true for two disjoint copies of $P$. 

\begin{conj}
If $P$ is the paw graph, then $P\sqcup P$ is common.
\end{conj}

If true, then it would be particularly interesting to have a flag algebra-free proof of this; however, beggars can't be choosers. 

It would be interesting to extend the results in this paper to other graphs. Given a graph $H$, define
\[DE(H):=\inf\{\ell\geq0: H\sqcup (\ell\cdot K_2)\text{ is uncommon}\}.\]
Here, $DE$ stands for ``disjoint edges.'' Note that, if $H\sqcup (\ell\cdot K_2)$ is common for all $\ell\geq0$, then $DE(H)=\infty$. Theorems~\ref{th:DE(K3cupK3)},~\ref{th:5/3} and~\ref{th:uncommon} can be viewed as bounds on $DE(k\cdot K_3)$ for certain values of $k$. We ask the following. 

\begin{ques}
Does the sequence $(DE(k\cdot K_3)/k)_{k=1}^\infty$ converge? If so, what is its limit?
\end{ques}

It would also be interesting to compute or bound $DE(H)$ for families graphs which are not covered by the results in this paper. In particular, the results proved in Section~\ref{sec:K3} may extend nicely to longer odd cycles. Some, but not all, of the ingredients needed for this already exist in the literature. Specifically, the fact that odd cycles are strongly common was proven very recently in~\cite{KimLee24} (the case of the $5$-cycle was proven in~\cite{Behague+Second}) and a correlation inequality generalizing Lemma~\ref{lem:tt} can be derived from a result of~\cite{Behague+Second}. A supersaturation theorem for $C_5$ in graphons of specific edge densities was obtained in~\cite{Bennet+20}, but the supersaturation problem for odd cycles is not well understood in general; this is the main barrier in generalizing our results on $K_3$ to longer odd cycles. 

The quantity $DE(H)$ is likely to be closely related to the function $\UC(H)$ defined in~\cite{GrzesikLeeLidickyVolec22} to be the minimum number of edges in a tree $T$ such that gluing a pendant copy of $T$ to $H$ yields an uncommon graph; when no such tree exists, $\UC(H):=\infty$. However, as Proposition~\ref{prop:pendant} and the case $k=3$ of Theorem~\ref{th:5/3} show, adding pendant edges can have a different effect than adding disjoint edges; see Figure~\ref{fig:3K3} below. The authors of~\cite{GrzesikLeeLidickyVolec22} raise the question of bounding $\UC(H)$ for non-bipartite graphs. Theorem~\ref{th:ttuncommon} provides an upper bound on $\UC(H)$ when $H$ is, for example, a \emph{triangle-vertex tree}; i.e. a triangle-tree in which one is only allowed to glue together two triangles on a single vertex (not on an edge).

\begin{figure}[htbp]
\begin{center}

\begin{tikzpicture}[scale=0.6]
  % define vertices with black circles
  \node[circle,fill,inner sep=1.5pt, minimum size=1mm] (A) at (330:1) {};
  \node[circle,fill,inner sep=1.5pt, minimum size=1mm] (B) at (90:1) {};
  \node[circle,fill,inner sep=1.5pt, minimum size=1mm] (C) at (210:1) {};
  \begin{scope}[shift={(2.5,0)}]
  \node[circle,fill,inner sep=1.5pt, minimum size=1mm] (D) at (330:1) {};
  \node[circle,fill,inner sep=1.5pt, minimum size=1mm] (E) at (90:1) {};
  \node[circle,fill,inner sep=1.5pt, minimum size=1mm] (F) at (210:1) {};
  \end{scope}
  \begin{scope}[shift={(5,0)}]
  \node[circle,fill,inner sep=1.5pt, minimum size=1mm] (G) at (330:1) {};
  \node[circle,fill,inner sep=1.5pt, minimum size=1mm] (H) at (90:1) {};
  \node[circle,fill,inner sep=1.5pt, minimum size=1mm] (I) at (210:1) {};
  \end{scope}
  \begin{scope}[shift={(6.5,0)}]
  \node[circle,fill,inner sep=1.5pt, minimum size=1mm] (J) at (270:0.5) {};
  \node[circle,fill,inner sep=1.5pt, minimum size=1mm] (K) at (90:1) {};
  \end{scope}
  \begin{scope}[shift={(7.5,0)}]
  \node[circle,fill,inner sep=1.5pt, minimum size=1mm] (L) at (270:0.5) {};
  \node[circle,fill,inner sep=1.5pt, minimum size=1mm] (M) at (90:1) {};
  \end{scope}
  \begin{scope}[shift={(8.5,0)}]
  \node[circle,fill,inner sep=1.5pt, minimum size=1mm] (N) at (270:0.5) {};
  \node[circle,fill,inner sep=1.5pt, minimum size=1mm] (O) at (90:1) {};
  \end{scope}
  \begin{scope}[shift={(9.5,0)}]
  \node[circle,fill,inner sep=1.5pt, minimum size=1mm] (P) at (270:0.5) {};
  \node[circle,fill,inner sep=1.5pt, minimum size=1mm] (Q) at (90:1) {};
  \end{scope}
  \begin{scope}[shift={(10.5,0)}]
  \node[circle,fill,inner sep=1.5pt, minimum size=1mm] (R) at (270:0.5) {};
  \node[circle,fill,inner sep=1.5pt, minimum size=1mm] (S) at (90:1) {};
  \end{scope}
  % draw edges
  \draw (A) -- (B) -- (C) -- (A);
  \draw (D) -- (E) -- (F) -- (D);
  \draw (G) -- (H) -- (I) -- (G);
  \draw (J)--(K);
  \draw (L)--(M);
  \draw (N)--(O);
  \draw (P)--(Q);
  \draw (R)--(S);
  
  \node (H1) at (5.25,-1.5) {$(3\cdot K_3)\sqcup(5\cdot K_2)$};
\end{tikzpicture}
\vspace{1em}

\begin{tikzpicture}[scale=0.6]
  % define vertices with black circles
  \node[circle,fill,inner sep=1.5pt, minimum size=1mm] (A) at (330:1) {};
  \node[circle,fill,inner sep=1.5pt, minimum size=1mm] (B) at (90:1) {};
  \node[circle,fill,inner sep=1.5pt, minimum size=1mm] (C) at (210:1) {};
  \node[circle,fill,inner sep=1.5pt, minimum size=1mm] (N) at (90:2) {};
  \begin{scope}[shift={(2.5,0)}]
  \node[circle,fill,inner sep=1.5pt, minimum size=1mm] (D) at (330:1) {};
  \node[circle,fill,inner sep=1.5pt, minimum size=1mm] (E) at (90:1) {};
  \node[circle,fill,inner sep=1.5pt, minimum size=1mm] (F) at (210:1) {};
  \node[circle,fill,inner sep=1.5pt, minimum size=1mm] (O) at (90:2) {};
  \end{scope}
  \begin{scope}[shift={(5,0)}]
  \node[circle,fill,inner sep=1.5pt, minimum size=1mm] (G) at (330:1) {};
  \node[circle,fill,inner sep=1.5pt, minimum size=1mm] (H) at (90:1) {};
  \node[circle,fill,inner sep=1.5pt, minimum size=1mm] (I) at (210:1) {};
  \node[circle,fill,inner sep=1.5pt, minimum size=1mm] (P) at (90:2) {};
  \end{scope}
  \begin{scope}[shift={(6.5,0)}]
  \node[circle,fill,inner sep=1.5pt, minimum size=1mm] (J) at (270:0.5) {};
  \node[circle,fill,inner sep=1.5pt, minimum size=1mm] (K) at (90:1) {};
  \end{scope}
  \begin{scope}[shift={(7.5,0)}]
  \node[circle,fill,inner sep=1.5pt, minimum size=1mm] (L) at (270:0.5) {};
  \node[circle,fill,inner sep=1.5pt, minimum size=1mm] (M) at (90:1) {};
  \end{scope}
  % draw edges
  \draw (A) -- (B) -- (C) -- (A);
  \draw (D) -- (E) -- (F) -- (D);
  \draw (G) -- (H) -- (I) -- (G);
  \draw (J)--(K);
  \draw (L)--(M);
  \draw (B)--(N);
  \draw (E)--(O);
  \draw (H)--(P);
  
  \node (H2) at (3.75,-1.5) {$(3\cdot P)\sqcup(2\cdot K_2)$};
\end{tikzpicture}
\vspace{-1em}

\end{center}
    \caption{The graph $(3\cdot K_3)\sqcup(5\cdot K_2)$ is common by Theorem~\ref{th:5/3} and the graph $(3\cdot P)\sqcup(2\cdot K_2)$ is uncommon by Proposition~\ref{prop:pendant}. }
    \label{fig:3K3}
\end{figure}
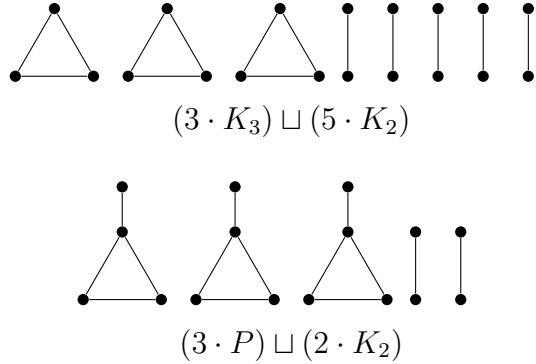

%A result of Jagger, \v{S}tov\'i\v{c}ek and Thomason~\cite{JaggerStovicekThomason96} says that $\UC(H)$ is finite for every non-bipartite graph $H$. Thus, adding a pendant tree to a non-bipartite graph is a well-known method for producing an uncommon graph. Also, every graph containing a $K_4$ is automatically uncommon by another result of~\cite{JaggerStovicekThomason96}. However, it seems that not much is known about uncommon graphs of chromatic number three and large minimum degree. Using Theorem~\ref{th:ttuncommon}, we can get examples with minimum degree two but, beyond that, we do not know

An inherent limitation of Lemma~\ref{lem:reduction} is that we do not know of many graphs $H$ which are $g$-bounded for a reasonably large function $g$. In particular, all of the known examples of strongly common graphs are Sidorenko graphs and odd cycles. It would be interesting to find new examples of non-bipartite strongly common graphs (if they exist) or examples of graphs that are $g$-bounded for other (non-trivial) functions $g$.

The proof of Theorem~\ref{th:uncommon} actually establishes a stronger statement. Specifically, if $H=k\cdot K_3\sqcup \ell\cdot K_2$ where $k\geq1$ and $\ell =\lceil1.9665k\rceil$, then there is a graphon $W$ such that
\[\max\{t(H,W),t(H,1-W)\}\leq(1/2)^{e(H)}\]
\[\min\{t(H,W),t(H,1-W)\}< (1/2)^{e(H)}.\]
If $H$ satisfies these conditions, one immediately gets that the disjoint union of any positive number of copies of $H$ is uncommon. It would be interesting to get a better understanding of the class of graphs $H$ with the property that such a graphon $W$ exists.

\begin{rem}
After submitting our paper to arxiv, we learned that Chen and Ma~\cite{ChenMa23+} obtained a much more general result than Theorem~\ref{th:commonNotStrongly}. Specifically, they prove that every graph containing a triangle, except for $K_3$ itself, is not strongly common. This was generalized by Versteegen~\cite{Versteegen23+} who proved that a graph of odd girth is strongly common if and only if it is an odd cycle. 
\end{rem}

\bibliographystyle{plain}

\appendix

\section{Simpler But Less Powerful Sufficient Conditions}
\label{appendix}

The inequalities \eqref{eq:ygood} and \eqref{eq:ybad} seem to be fairly unwieldy in general. Our goal in this appendix is to obtain sufficient conditions for these inequalities that may be easier to analyze in practice. To simplify \eqref{eq:ybad} for $\ell\geq0$, we apply the following simple inequality for real numbers.

\begin{lem}
\label{lem:numbers}
Let $a,b,c,d,s,t\geq0$ such that $b\geq d$ and $s\geq t$. Then
\[\left(ab^s-cd^s\right)\left(b^t+d^t\right)\geq \left(ab^t-cd^t\right)\left(b^s+d^s\right).\]
\end{lem}

\begin{proof}
Expand both sides of the inequality. After doing a bit of cancellation, it becomes
\[ab^sd^t-cd^sb^t\geq ab^td^s-cd^tb^s\]
which is the same as
\[ab^sd^t+cd^tb^s\geq ab^td^s+cd^sb^t.\]
Since $s\geq t$ and $b\geq d$, we have $ab^sd^t\geq ab^td^s$ and $cd^tb^s\geq cd^sb^t$; so, the inequality holds. 
\end{proof}

\begin{lem}
\label{lem:simplifyybad}
Let $k,\ell,g$ and $\rho$ satisfy the hypotheses of Proposition~\ref{prop:reduction}. For $0\leq x\leq 1$, if $\ell\geq\ell_0$ and 
\[\frac{g(1+x)(1+x)^{\frac{\ell_0}{k-1}}-g(1-x)(1-x)^{\frac{\ell_0}{k-1}}}{(1+x)^{\frac{\ell_0}{k-1}}+ (1-x)^{\frac{\ell_0}{k-1}}}\geq g(1+x)-\rho(1+x)\]
then \eqref{eq:ybad} holds. In particular, if $\ell\geq 0$ and
\begin{equation}
\label{eq:ybadweaker}
2\rho(1+x)\geq g(1+x)+g(1-x),
\end{equation}
then \eqref{eq:ybad} holds. 
\end{lem}

\begin{proof}
Applying Lemma~\ref{lem:numbers} with $a=g(1+x)$, $b=(1+x)^{\frac{1}{k-1}},c=g(1-x),d=(1-x)^{\frac{1}{k-1}},s=\ell$ and $t=\ell_0$ yields
\[\frac{g(1+x)(1+x)^{\frac{\ell}{k-1}} -g(1-x)(1-x)^{\frac{\ell}{k-1}}}{(1+x)^{\frac{\ell}{k-1}} + (1-x)^{\frac{\ell}{k-1}}}\geq \frac{g(1+x)(1+x)^{\frac{\ell_0}{k-1}} -g(1-x)(1-x)^{\frac{\ell_0}{k-1}}}{(1+x)^{\frac{\ell_0}{k-1}} + (1-x)^{\frac{\ell_0}{k-1}}}.\]
Thus, if the right side of this inequality is at least $g(1+x)-\rho(1+x)$, then so is the left. 

For the ``in particular'' part of the lemma, if $\ell\geq0$, then we can apply the first part of the lemma with $\ell_0=0$ to get that the following inequality implies \eqref{eq:ybad}:
\[\frac{g(1+x)-g(1-x)}{2}\geq g(1+x)-\rho(1+x).\]
Rearranging this inequality gives us \eqref{eq:ybadweaker}.
\end{proof}

Next, we use H\"older's Inequalityto find simpler sufficient conditions for \eqref{eq:ygood}.

\begin{lem}[H\"older's Inequality, see~{\cite[Section~2.8]{Mitrinovic70}}]
\label{lem:Holder}
For $n\geq1$, if $x_1,\dots,x_n,y_1,\dots,y_n$ are positive real numbers and $p,q>1$ such that $\frac{1}{p}+\frac{1}{q}=1$, then
\[\sum_{i=1}^nx_iy_i\leq \left(\sum_{i=1}^nx_i^p\right)^{1/p}\left(\sum_{i=1}^ny_i^q\right)^{1/q}.\]
\end{lem}

\begin{claim}
\label{claim:Holder}
For any non-negative real numbers $b_1,b_2,s,m$ such that $m\geq s>0$,
\[\left(b_1^{s} + b_2^{s}\right)^m\leq 2^{m-s}\left(b_1^m+b_2^m\right)^{s}.\]
\end{claim}

\begin{proof}
If $m=s$, then both sides of the inequality evaluate to $\left(b_1^{s} + b_2^{s}\right)^m$ and there is nothing to prove; so, assume that $m>s$. Define $p=m/s$ and $q=m/(m-s)$ and note that $\frac{1}{p}+\frac{1}{q}=1$. Let $x_1=b_1^s, x_2=b_2^s$ and $y_1=y_2=1$. By Lemma~\ref{lem:Holder},
\[x_1y_1+x_2y_2\leq \left(x_1^p+x_2^p\right)^{1/p}\left(y_1^q+y_2^q\right)^{1/q}.\]
In other words, 
\[b_1^s+b_2^s\leq \left(b_1^m+b_2^m\right)^{s/m}\left(1^{m/(m-s)}+1^{m/(m-s)}\right)^{(m-s)/m}.\]
The desired inequality now follows by raising both sides to the power $m$. 
\end{proof}

\begin{lem}
\label{lem:simplifyygood}
Let $k,\ell,g$ and $\rho$ satisfy the hypotheses of Proposition~\ref{prop:reduction}. For $0\leq x< 1$, if $\ell_0\geq\ell$ and 
\begin{equation}
\label{eq:ygoodweaker}
(g(1+x)+g(1-x))^k \geq \frac{c2^{k-2}\left((1+x)^{\ell_0} + (1-x)^{\ell_0}\right)}{(1+x)^{\ell_0}(1-x)^{\ell_0}}
\end{equation}
then \eqref{eq:ygood} holds. 
\end{lem}

\begin{proof}
Applying Claim~\ref{claim:Holder} with $s=\frac{\ell}{k-1},m=\ell, b_1=1+x$ and $b_2=1-x$ yields
\[\left((1+x)^{\frac{\ell}{k-1}}+(1-x)^{\frac{\ell}{k-1}}\right)^{\ell}\leq 2^{\ell-\frac{\ell}{k-1}}\left((1+x)^\ell+(1-x)^\ell\right)^{\frac{\ell}{k-1}}.\]
Raising both sides to the power $\frac{k-1}{\ell}$ gives us
\[\left((1+x)^{\frac{\ell}{k-1}}+(1-x)^{\frac{\ell}{k-1}}\right)^{k-1}\leq 2^{k-2}\left((1+x)^\ell+(1-x)^\ell\right).\]
Thus, to prove \eqref{eq:ygood}, it is sufficient to show that
\[\left(g(1+x)+g(1-x)\right)^k\geq \frac{c2^{k-2}\left((1+x)^\ell+(1-x)^\ell\right)}{(1+x)^{\ell}(1-x)^{\ell}}.\]
Note that the function $\frac{(1+x)^\ell+(1-x)^\ell}{(1+x)^{\ell}(1-x)^{\ell}}$ is increasing in $\ell$. So, for $\ell_0\geq \ell$, the inequality \eqref{eq:ygoodweaker} is also sufficient for proving \eqref{eq:ygood}. 
\end{proof}

\begin{lem}
\label{lem:simplifyygood2}
Let $k,\ell,g$ and $\rho$ satisfy the hypotheses of Proposition~\ref{prop:reduction}. Furthermore, let $m\geq\frac{\ell}{k-1}$ and assume that $g(z)=z^m$ for $z\in [0,2]$. For $0\leq x< 1$, if 
\begin{equation}
\label{eq:ygoodweaker2}
\left(\frac{(1+x)^m+(1-x)^m}{2}\right)^{k-\frac{\ell}{m}}\geq \frac{c}{2(1+x)^{\ell}(1-x)^{\ell}}
\end{equation}
then \eqref{eq:ygood} holds. 
\end{lem}

\begin{proof}
Applying Claim~\ref{claim:Holder} with $s=\frac{\ell}{k-1}, b_1=1+x$ and $b_2=1-x$ yields
\[\left((1+x)^{\frac{\ell}{k-1}}+(1-x)^{\frac{\ell}{k-1}}\right)^{m}\leq 2^{m-\frac{\ell}{k-1}}\left((1+x)^m+(1-x)^m\right)^{\frac{\ell}{k-1}}.\]
Raising both sides to the power $\frac{k-1}{m}$ gives us
\[\left((1+x)^{\frac{\ell}{k-1}}+(1-x)^{\frac{\ell}{k-1}}\right)^{k-1}\leq 2^{k-1 -\frac{\ell}{m}}\left((1+x)^m+(1-x)^m\right)^{\frac{\ell}{m}}.\]
Thus, to prove \eqref{eq:ygood}, it is sufficient to show that
\[\frac{\left((1+x)^m+(1-x)^m\right)^k}{2^{k-1 -\frac{\ell}{m}}\left((1+x)^m+(1-x)^m\right)^{\frac{\ell}{m}}}\geq \frac{c}{(1+x)^{\ell}(1-x)^{\ell}}.\] 
This inequality is equivalent to \eqref{eq:ygoodweaker2}. 
\end{proof}

\end{document}